\documentclass[10pt,english,a4paper]{article}
\pdfoutput=1
\makeatletter{}%
\usepackage{ifpdf}   %
\usepackage{multido} %
\usepackage{ifthen}

\usepackage{color}

\ifpdf
\usepackage[pdftex]{graphicx}
\else
\usepackage[dvips]{graphicx}
\fi

\usepackage{float,floatflt,afterpage}
\usepackage{wrapfig}

\usepackage{etex}
\usepackage{tikz}
\usetikzlibrary{snakes}

\usepackage[arrow,matrix,curve]{xy} %
\usepackage{subfigure} %

\usepackage{babel}
\usepackage[latin1]{inputenc}

\usepackage[centertags,intlimits]{amsmath}
\usepackage{amsthm}
\usepackage{amssymb}
\usepackage{mathrsfs}

\theoremstyle{break}
\allowdisplaybreaks[1]

\usepackage{exscale}
\usepackage{dsfont}
\usepackage{bbm}

\ifpdf
\usepackage[pdftex,
  bookmarks,
  bookmarksopen=true,
  bookmarksnumbered=true,
  backref,
  linkcolor=blue,
  urlcolor=blue,
  pdfauthor={Andreas Fischle and Patrizio Neff},
  pdftitle={The geometrically nonlinear Cosserat micropolar shear-stretch
       energy. \mbox{Part I:}
       A general parameter reduction formula and energy-minimizing
       microrotations in 2D},
  hyperfigures=true,
  pdfpagelabels,
  pagebackref,
  hypertexnames=true,
  plainpages=false,
  naturalnames]{hyperref}%
\else
\usepackage[dvips,breaklinks]{hyperref} %
\usepackage{breakurl}
\fi

\usepackage{url} %
\usepackage{breakurl} %

\makeatletter
\def\url@leostyle{%
  \@ifundefined{selectfont}{\def\UrlFont{\sf}}{\def\UrlFont{\small\ttfamily}}}
\makeatother
\usepackage{enumerate}
\usepackage{verbatim}        %
\usepackage{lscape}          %
\usepackage{array,longtable} %
\usepackage{listings}        %

\usepackage{makeidx}

\makeatletter{}%
\def\and{{\rm and}}          %
\def\cof#1{{\rm Cof}\,{#1}}  %
\def\det#1{{\rm det}\,{#1}}  %

\def\skew{{\rm skew}}        %
\newfont{\Sf}{cmssbx10 scaled 2074}
\newbox{\assem}
\savebox{\assem}(15,2)[bl]{
\begin{normalsize}
\unitlength=1.0mm
\begin{picture}(6.00,0.00)
\put(.0,1.0){\makebox(0,0)[cc]{$\cup$}}
\put(.0,-1.0){\makebox(0,0)[cc]{$\mbox{\scriptsize \it n=1}$}}
\put(.0,4.0){\makebox(0,0)[cc] {\scriptsize $\it n_{\mbox{\tiny\it elm}}$}}
\end{picture}
\end{normalsize}}

\newbox{\asse}
\savebox{\asse}(15,2)[bl]{
\begin{normalsize}
\unitlength=1.0mm
\begin{picture}(8.00,0.00)
\put(.0,2.0){\makebox(0,0)[cc]{\Sf A}}
\end{picture}
\end{normalsize}}

\def\sqtwo3{{\textstyle {\sqrt{2 \over 3}}}}   %
\newcommand{\IP}{{\rm I\kern-.18em P}}           %
\newcommand{\II}{{\rm I\kern-.18em I}}           %
\newcommand{\IF}{{\rm I\kern-.25em F}}           %
\newcommand{\IE}{{\rm I\kern-.25em E}}           %
\def\IR{{\rm I\kern-.15em R}}

\newcommand{\ia}{{\rm\kern.24em                  %
   \vrule width.02em height0.9ex depth-.05ex
   \kern-.26em a}}
\newcommand{\ic}{{\rm\kern.24em                  %
   \vrule width.02em height0.9ex depth-.05ex
   \kern-.26em c}}
\newcommand{\IA}{{\rm\kern.22em                  %
    \vrule width.02em
        height0.5ex depth 0ex
    \kern-.24em A}}
\newcommand{\IC}{{\rm\kern.24em                  %
   \vrule width.02em height1.4ex depth-.05ex
   \kern-.26em C}}
\makeatletter{}%
\DeclareMathOperator{\sign}{sign}
\DeclareMathOperator{\dist}{dist}

\newcommand{\malpha}{\alpha}
\newcommand{\palpha}{\alpha_\text{p}}
\renewcommand{\epsilon}{\varepsilon}

\newcommand{\norm}[1]{\|#1\|}
\newcommand{\abs}[1]{\left| #1 \right|}
\newtheorem{lem}{Lemma}[section]

\newtheorem{rem}[lem]{Remark}
\newtheorem{defi}[lem]{Definition}

\newtheorem{theo}[lem]{Theorem}

\newtheorem{cor}[lem]{Corollary}

\newtheorem{prob}[lem]{Problem}

\newcommand{\leref}[1]{Lemma \ref{#1}}
\newcommand{\theref}[1]{Theorem \ref{#1}}

\newcommand{\coref}[1]{Corollary \ref{#1}}

\newcommand{\deref}[1]{Definition \ref{#1}}

\newcommand{\probref}[1]{Problem \ref{#1}}

\newcommand{\R}{\mathbb{R}}

\newcommand{\C}{\mathbb{C}}

\DeclareMathOperator{\GL}{GL}
\DeclareMathOperator{\SO}{SO}

\DeclareMathOperator{\skewop}{skew}
\renewcommand{\skew}{\skewop}
\DeclareMathOperator{\diag}{diag}
\DeclareMathOperator{\sym}{sym}
\DeclareMathOperator{\Tr}{tr}

\DeclareMathOperator{\so}{\mathfrak{so}}
\DeclareMathOperator{\gl}{\mathfrak{gl}}

\DeclareMathOperator{\polar}{R_{\rm p}}

\newcommand{\Sym}{ {\rm{Sym}} }
\newcommand{\Psym}{ {\rm{PSym}} }
\renewcommand{\qed}{\hfill $\blacksquare$\\}

\newcommand{\id}{{\boldsymbol{\mathbbm{1}}}}

\DeclareMathOperator{\Det}{det}

\renewcommand{\det}[1]{ {\Det[{#1}]} }
\newcommand{\tr}[1]{ {\Tr \left[{#1}\right]} }
\makeatletter{}%
\newcommand{\secref}[1]{Section \ref{#1}}
\newcommand{\figref}[1]{Figure \ref{#1}}

\definecolor{orange}{rgb}{1.0,0.5,0}

\DeclareMathOperator{\Reals}{\mathbb{R}}
\renewcommand{\R}{\Reals}

\DeclareMathOperator{\argminmathop}{argmin}
\newcommand{\argmin}[2]{\mathchoice{\underset{#1}{\argminmathop}\, {#2}}{\argminmathop_{#1}\, {#2}}{}{}}

\newcommand{\scalprod}[2]{\big<#1,\,#2\big>}

\newcommand{\setdef}[2]{\lbrace #1 \;\vert\; #2\rbrace}

\newcommand{\ddtat}[2]{\frac{\rm d}{\rm dt}\,#1\big{\vert}_{t = #2}}

\newcommand{\hsnorm}[1]{\left\lVert #1 \right\rVert}

\DeclareMathOperator{\RPosZ}{\sideset{}{_0^+}\Reals}

\DeclareMathOperator{\eqdef}{\,\mathrel{\mathop:}=\,}

\DeclareMathOperator{\isequivto}{\,\Longleftrightarrow\,}
\makeatletter{}%
\newcommand{\mrot}{R}
\newcommand{\mstretch}{\overline{U}}

\DeclareMathOperator{\rpolar}{rpolar}

\DeclareMathOperator{\sradmm}{\rho_{\mu,\,\mu_c}}
\DeclareMathOperator{\sradmmdef}{\frac{2\,\mu}{\mu\,-\,\mu_c}}

\DeclareMathOperator{\wmm}{W_{\mu,\mu_c}}
\DeclareMathOperator{\wsym}{W_{1,0}}

\DeclareMathOperator{\wmmtilde}{\widetilde{W}_{\mu,\,\mu_c}}

\DeclareMathOperator{\wsymtilde}{\widetilde{W}_{1,0}}

\newcommand{\countres}{
  \setcounter{equation}{0}
  \setcounter{figure}{0}
  \setcounter{table}{0}
}

\renewcommand{\theequation}{\arabic{section}.\arabic{equation}}

\renewcommand{\baselinestretch}{1.0}          %
\makeatletter

\renewcommand{\itemize}{%
  \ifnum \@itemdepth >\thr@@\@toodeep\else
    \advance\@itemdepth\@ne
    \edef\@itemitem{labelitem\romannumeral\the\@itemdepth}%
    \expandafter
    \list
      \csname\@itemitem\endcsname
      {\def\makelabel##1{\hss\llap{##1}}%
        \topsep=.8ex\itemsep=-.2ex}%
  \fi}

\renewcommand\section{\@startsection {section}{1}{\z@}%
  {-3.5ex \@plus -1ex \@minus -.2ex}%
  {2.3ex \@plus.2ex}%
  {\boldmath\normalfont\Large\bfseries}}
\renewcommand\subsection{\@startsection{subsection}{2}{\z@}%
  {-3.25ex\@plus -1ex \@minus -.2ex}%
  {1.5ex \@plus .2ex}%
  {\boldmath\normalfont\large\bfseries}}
\renewcommand\subsubsection{\@startsection{subsubsection}{3}{\z@}%
  {-3.25ex\@plus -1ex \@minus -.2ex}%
  {1.5ex \@plus .2ex}%
  {\boldmath\normalfont\normalsize\bfseries}}
\renewcommand\paragraph{\@startsection{paragraph}{4}{\z@}%
  {3.25ex \@plus1ex \@minus.2ex}%
  {-1em}%
  {\boldmath\normalfont\normalsize\bfseries}}
\renewcommand\subparagraph{\@startsection{subparagraph}{5}{\parindent}%
  {3.25ex \@plus1ex \@minus .2ex}%
  {-1em}%
  {\boldmath\normalfont\normalsize\bfseries}}

\makeatother
\title{The geometrically nonlinear Cosserat micropolar shear-stretch
       energy. \mbox{Part I:}
       A general parameter reduction formula and energy-minimizing
       microrotations in 2D}

\author{Andreas Fischle
\!\!\thanks{Corresponding author: Andreas Fischle,
Fakult\"at f\"ur Mathematik, Informatik, Naturwissenschaften,
IGPM, RWTH Aachen,
Templergraben 55,
52062 Aachen,
Germany,
email: fischle@igpm.rwth-aachen.de}
\quad and \quad
Patrizio Neff
\!\!\thanks{Patrizio Neff,
Head of Lehrstuhl f\"{u}r Nichtlineare Analysis und Modellierung,
Fakult\"{a}t f\"{u}r Mathematik,
Universit\"{a}t Duisburg-Essen,
Thea-Leymann Str. 9,
45127 Essen,
Germany,
email: patrizio.neff@uni-due.de}}

\begin{document}
\topmargin-3.0cm  %
\selectfont
\vspace{-4cm}
\maketitle

\pagenumbering{arabic}

\makeatletter{}%
\begin{center}\textbf{Abstract}\end{center}
\begin{center}
  \begin{minipage}{0.95\textwidth}
    In any geometrically nonlinear quadratic Cosserat-micropolar
    extended continuum model formulated in the deformation gradient field
    $F \eqdef \nabla\varphi: \Omega \to \GL^+(n)$
    and the microrotation field $\mrot: \Omega \to \SO(n)$,
    the shear-stretch energy is necessarily of the form
    \begin{align*}
      \wmm(\mrot\,;F)& \eqdef
      \mu  \hsnorm{\sym(\mrot^TF - \id)}^2
      + \mu_c\hsnorm{\skew(\mrot^TF - \id)}^2\;,
    \end{align*}
    where $\mu > 0$ is the Lam\'e shear modulus and $\mu_c \geq 0$
    is the Cosserat couple modulus. In the present contribution, we
    work towards explicit characterizations of the set of optimal
    Cosserat microrotations
    $\argmin{\mrot\,\in\,\SO(n)}{\wmm(\mrot\,;F)}$ as a function
    of $F \in \GL^+(n)$ and weights $\mu > 0$ and $\mu_c \geq 0$.
    For $n \geq 2$, we prove a parameter
    reduction lemma which reduces the optimality problem
    to two limit cases: $(\mu, \mu_c) = (1,1)$ and $(\mu,\mu_c) = (1,0)$.
    In contrast to Grioli's theorem,
    we derive non-classical minimizers for the parameter range
    $\mu > \mu_c \geq 0$ in dimension $n\!=\!2$. Currently, optimality
    results for $n \geq 3$ are out of reach for us, but we contribute
    explicit representations for $n\!=\!2$ which we name
    $\rpolar^{\pm}_{\mu,\mu_c}(F) \in \SO(2)$ and which
    arise for $n\!=\!3$ by fixing the rotation axis a priori. Further,
    we compute the associated reduced energy levels and study the
    non-classical optimal Cosserat rotations
    $\rpolar^\pm_{\mu,\mu_c}(F_\gamma)$ for simple planar shear.
  \end{minipage}
\end{center}

\vspace*{0.5cm}

\vspace*{0.25cm}
%
%
%

%
%
%
%
%
%
%
%
%
%
%

%
 %
\countres

\setcounter{tocdepth}{1}
\renewcommand{\baselinestretch}{0}\normalsize
\tableofcontents
\renewcommand{\baselinestretch}{1.0}\normalsize
\topmargin-1.0cm  %
\selectfont

\newpage
\makeatletter{}%
\section{Introduction}\label{sec:intro}
In 1940 Guiseppe Grioli proved a remarkable optimality
result~\cite{Grioli40} for the polar factor in dimension $n = 3$. To
state his result, we denote by $\polar(F) \in \SO(n)$ the unique
orthogonal factor of $F \in \GL^+(n)$ in the right polar
decomposition $F = \polar(F)\,U(F)$ and by
$U(F) = \polar(F)^TF = \sqrt{F^TF} \in \Psym(n)$
the symmetric positive definite Biot stretch tensor.
In~\cite{Grioli40} Grioli proved
the special case $n = 3$ of the following theorem:
\begin{theo}[Grioli's theorem~\cite{Grioli40,Guidugli80,Pietraszkiewicz05}]
  \label{theo:intro:grioli}
  Let $n \geq 2$ and $\hsnorm{X}^2 \eqdef \tr{X^TX}$ the Frobenius norm.
  Then for any $F \in \GL^+(n)$, it holds
\begin{equation}
  \argmin{\mrot\,\in\,\SO(n)}{\hsnorm{\mrot^TF - \id}^2} \;=\; \{\polar(F)\},
  \quad\text{and thus}\quad
  \min_{\mrot\,\in\,\SO(n)}{\hsnorm{\mrot^TF - \id}^2}
  \;=\; \hsnorm{U - \id}^2\;.\label{intro:eq:min_energy_simple}
\end{equation}
\end{theo}
The optimality of the polar factor $\polar(F)$ for $n = 3$ generalizes
to any dimension $n \geq 2$, see, e.g.,~\cite{Guidugli80}, and it is
this more general theorem to which we shall refer as Grioli's theorem
in this present work. A modern exposition of the original contribution
of Grioli has been recently made available in~\cite{Neff_Grioli14}.

In contrast to Grioli's theorem, it has been noted in~\cite{Neff_Biot07} that
the polar factor $\polar(F)$ is \emph{not} necessarily optimal for a more
general formulation of Grioli's theorem with weights. Hence, our main
objective is to make progress on
\begin{prob}[Weighted optimality]
Let $n \geq 2$. Compute the set of optimal rotations
\label{intro:prob:weighted}
\begin{equation}
  \argmin{\mrot\,\in\,\SO(n)}{\wmm(\mrot\,;F)} \eqdef
  \argmin{\mrot\,\in\,\SO(n)}{%
    \left\{
    \mu\, \hsnorm{\sym(\mrot^TF - \id)}^2
    \,+\,
    \mu_c\,\hsnorm{\skew(\mrot^TF - \id)}^2\right\}
  }
  \label{eq:intro:weighted}
\end{equation}
for given $F \in \GL^+(n)$ and weights $\mu > 0,\mu_c \geq 0$
such that $\mu \neq \mu_c$. Here, $\sym(X) \eqdef \frac{1}{2}(X + X^T)$
and $\skew(X) \eqdef \frac{1}{2}(X - X^T)$.
\end{prob}
For $\mu = \mu_c$, we recover Grioli's theorem. This case is well understood.
However, no systematic analysis of the proposed weighted optimality problem
seems to exist in the literature.

Note that~\probref{intro:prob:weighted} is of independent interest in
the mechanics of micropolar media and Cosserat theory in particular.
The weighted strain energy contribution $\wmm$ arises in any geometrically
nonlinear Cosserat-micropolar model proposed in, e.g.,~\cite{Neff_Cosserat_plasticity05},
cf.~\cite{Boehmer:2015:SS,Eremeyev:2012:FMM,Lankeit:2015:IC,Sansour:2008:NCC,Neff01d,Neff_Cosserat_plasticity05,Pietraszkiewicz:2009:VPN} and for any 6-parameter geometrically
nonlinear Cosserat shell model proposed
in~\cite{Birsan_Neff_MMS_2013,Wisniewski98,Wisniewski00,Wisniewski02}.
In the beforementioned material models $\wmm$ determines the shear-stretch
energy contribution. To obtain a full Cosserat continuum model, $\wmm$
needs to be augmented by a curvature energy term~\cite{Neff_curl06} and a
volumetric energy term, see, e.g.,~\cite{Neff:2015:EGNC} or~\cite{Neff_Biot07}.
However,~\probref{intro:prob:weighted} reappears as a limit case for vanishing
characteristic length $L_{\rm c} = 0$.\footnote{For this interpretation to
  work, we have to assume that the volume term decouples from the
  \mbox{microrotation $\mrot$}, e.g.,
$W^{\rm vol}(\mstretch)
  \eqdef \frac{\lambda}{4} \left[\left(\det{\mstretch} - 1\right)^2 + \left(\frac{1}{\det{\mstretch}} - 1\right)^2\right]$, which seems natural, since it
  is trivially satisfied in the linear Cosserat models \cite{Neff_Jeong_bounded_stiffness09,Neff_ZAMM05,Neff_Jeong_Conformal_ZAMM08}. The
  linear elastic energy density is of the form
$W^{\rm lin}(\nabla u, A) \eqdef \mu\,\hsnorm{\sym(\nabla u - A)}^2
+ \mu_c\,\hsnorm{\skew(\nabla u - A)}^2
+ \frac{\lambda}{2}\,\tr{(\nabla u - A)}^2$.
Here, for \emph{all} admissible values of $\mu > 0$ and $\mu_c > 0$, the
unique minimizer is given by $A = \skew(\nabla u)$.
}
In this scenario, the curvature term vanishes identically and the set of
global minimizers for~\eqref{eq:intro:weighted} admits an interpretation
as optimal Cosserat rotations for given deformation gradient
$F \in \GL^+(n)$ and material parameters $\mu$ and $\mu_c$.
Grioli's theorem implies that for $\mu = \mu_c$ the optimal Cosserat rotation
is uniquely given by $\polar(F)$, i.e., it is the orthogonal part of the
right polar decomposition $F = \polar\,U$. The corresponding reduced energy
level is the Biot-energy, see~\cite{Neff_Biot07}.

The suitable choice of the so-called Cosserat couple modulus $\mu_c \geq 0$
for specific materials or boundary value problems is an interesting open
question~\cite{Neff_ZAMM05}. In particular,
a strictly positive choice $\mu_c > 0$ is debatable~\cite{Neff_ZAMM05} and a
better understanding of the limit case $\mu_c = 0$ is hence of interest. This
further motivates to study the weighted formulation stated
in~\probref{intro:prob:weighted}.

We want to stress that although the term $\wmm$ subject to minimization
in~\eqref{eq:intro:weighted} is quadratic in the nonsymmetric
microstrain tensor $\mstretch - \id = \mrot^TF - \id$
(see, e.g.,~\cite{Eremeyev:2012:FMM}), the associated
minimization problem with respect to $\mrot$ is nonlinear due to the
multiplicative coupling of $\mrot$ and $F$ and the geometry of $\SO(n)$.

\begin{rem}[Existence of global minimizers]
  The energy $\wmm(R\,;F)$ is a polynomial in the matrix entries,
  hence $\wmm \in C^\infty(\SO(n),\Reals)$. Further, since the Lie
  group $\SO(n)$ is compact and $\partial\!\SO(n) = \emptyset$, the
  global extrema of $\wmm$ are attained at interior points.
\end{rem}

The previous remark hints at a possible solution strategy for
\probref{intro:prob:weighted}. Suppose that we succeed to compute
\emph{all} the critical points $\mrot_\textrm{crit} \in \SO(n)$ of
$\wmm(\mrot\,;F)$.\footnote{Since the boundary of $\SO(n)$ is empty the
rotation $\mrot_\textrm{crit}$ is a critical point at $F \in \GL^+(n)$
if and only if $\ddtat{\wmm(\mrot(t)\,; F)}{0} = 0$ for every smooth
curve of rotations $\mrot(t): (-\epsilon, \epsilon) \to \SO(n)$
passing through $\mrot(0) = \mrot_\textrm{crit}$.} Then, if possible,
a direct comparison of the associated critical energy levels
$\wmm(\mrot_\textrm{crit}\,;F)$ might allow us to identify the
energy-minimizing branches for given parameters $F, \mu$ and $\mu_c$.
Note that any minimal branch coincides with the \emph{reduced}
Cosserat shear-stretch energy
\begin{align}
  W^{\rm red}_{\mu,\mu_c}: \GL^+(n) \to \RPosZ\;,\quad
  W^{\rm red}_{\mu,\mu_c}(F) \eqdef \min_{\mrot\,\in\,\SO(n)} W_{\mu,\mu_c}(\mrot\,;F)\;.
\end{align}

At present a solution for the three-dimensional problem (let alone the
$n$-dimensional problem) seems out of reach for us. Therefore, we restrict
our attention to the planar case where we can base our computations on
the standard parametrisation
\begin{equation}
  \mrot: [-\pi, \pi] \to \SO(2) \subset \R^{2\times 2},\quad \mrot(\malpha) \eqdef
  \begin{pmatrix}
    \cos \malpha & -\sin \malpha\\
    \sin \malpha & \cos \malpha
  \end{pmatrix}
\end{equation}
by a rotation angle.\footnote{Note that $\pi$ and $-\pi$ are mapped
to the same rotation. In this text, we implicitly choose $\pi$ over
$-\pi$ for the rotation angle whenever uniqueness is an issue.}

Regarding the material parameters, we prove that, for any dimension
$n \geq 2$, it is sufficient to restrict our attention to two parameter
pairs:
$(\mu,\mu_c) = (1,1)$, the \emph{classical} case, and
$(\mu,\mu_c) = (1,0)$, the \emph{non-classical} case.
We shall see that, somewhat surprisingly, the solutions for arbitrary
$\mu > 0$ and $\mu_c \geq 0$ can be recovered from these two limiting
cases. A large part of this paper is dedicated to the discussion of
the non-classical choice of material parameters in the planar case.

\begin{prob}[The planar minimization problem]
  \label{intro:prob:planar}
Let $F \in \GL^+(2)$, $\mu > 0$ and $\mu_c \geq 0$. The task is to compute
the set of optimal microrotation angles
\begin{equation}
  \argmin{\malpha\;\in\;[-\pi,\pi]}{\hsnorm{(\sqrt{\mu}\,\sym + \sqrt{\mu_c}\,\skew).\left[\begin{pmatrix}
      \cos \malpha & -\sin \malpha\\
      \sin \malpha & \cos \malpha
\end{pmatrix}^T
\begin{pmatrix}
    F_{11} & F_{12}\\
    F_{21} & F_{22}\\
\end{pmatrix}
-\begin{pmatrix}
 1 & 0\\
 0 & 1
\end{pmatrix}\right]}^2}\,.\label{eq:intro:planar}
\end{equation}
\end{prob}

It turns out that there are at most two optimal planar rotations
in what we will discern as the non-classical parameter range
$\mu > \mu_c \geq 0$. Both of them coincide with the polar factor
$\polar(F)$ in the compressive regime of $F \in \GL^+(2)$, but
deviate elsewhere; see~\secref{sec:minimization}. We denote the
explicit formulae for the optimal Cosserat rotations (i.e., solutions
to~\eqref{eq:intro:weighted}) by
$\rpolar_{\mu,\mu_c}(F)$ and the respective rotation angles
(i.e., solutions to~\eqref{eq:intro:planar})
by $\alpha_{\mu,\mu_c}(F)$ (possibly multi-valued).
The computation of global minimizers in dependence of $F$ is not
completely obvious even for the reduced planar case. We hope that
the discovered mechanisms will help to understand the cases
$n \geq 3$ eventually.

This paper is now structured as follows: we first prove a dimension-independent
parameter transformation lemma for $\mu$ and $\mu_c$ in~\secref{sec:reduction}
which allows us to focus on a classical $(\mu,\mu_c) = (1,1)$ and
a non-classical limit case $(\mu,\mu_c) = (1,0)$. In \secref{sec:minimization},
we compute the optimal planar Cosserat rotations and associated formally
reduced energies, first for the classical and non-classical limit cases
for $\mu$ and $\mu_c$, then for all admissible values of the material
parameters. Finally, the optimal rotations for the non-classical limit
case $(\mu,\mu_c) = (1,0)$ are specialized to the particular case of
planar simple shear in~\secref{sec:shear}. A short appendix provides
some elementary but useful matrix identities for $n = 2$.

 %
\countres
\makeatletter{}%
\section{Parameter reduction for \texorpdfstring{$\mu$}{mu} and \texorpdfstring{$\mu_c$}{mu_c}}\label{sec:reduction}

The Cosserat shear-stretch energy~\eqref{eq:intro:weighted} is parametrized
by two material parameters: $\mu > 0$ and $\mu_c \geq 0$. It can be shown
that $\mu$ must coincide with the classical Lam\'e shear modulus of linear
elasticity. The interpretation of the Cosserat couple modulus $\mu_c$ is,
however, less clear. It seems that all measurements that can
be found in the literature lead to inconsistencies for $\mu_c > 0$, see,
e.g.,~\cite{Neff_ZAMM05}, and so a better understanding of these
material parameters and their interaction is of interest. In this section,
we contribute some helpful representations of the Cosserat shear-stretch
energy in terms of powers of $\tr{\,\mstretch\,} = \tr{\mrot^TF}$.

Let us introduce the following equivalence relation for continuous
functions $f,g: X \to \R$ defined on a compact set $X$:
\begin{equation}
f \sim_X g
\quad\isequivto\quad
\argmin{x\;\in\;X}{f(x)} \;=\; \argmin{x\;\in\;X}{g(x)}\;.
\end{equation}
An important example is given by $f \sim_X \lambda f + c$
for $\lambda > 0$ and $c$ both independent of $x$.\footnote{Note that
  $f$ and $\lambda f + c$ share the same critical point structure,
  whereas $f \sim g$ does not imply this.} In the present
work, we consider minimization w.r.t. $X = \SO(n)$ and we shall
simply write $f \sim g$ instead of $f \sim_{\SO(n)} g$.

In this section, we show that it is sufficient to restrict our attention
to two representative pairs of parameters, \emph{the classical limit case
$(\mu,\mu_c) = (1,1)$} and \emph{the non-classical limit case $(\mu,\mu_c) = (1,0)$}.
The solutions for arbitrary admissible $\mu$ and $\mu_c$ can then
be recovered from these two limiting cases by a suitable transformation,
see~\secref{subsec:minimization:reconstruction} for details. To this end,
we first introduce the following
\begin{defi}[Parameter rescaling]
  \label{defi:reduction:srad}
  \label{defi:reduction:lambda}
  \label{defi:reduction:rescaling}
  Let $\mu > \mu_c \geq 0$. We define the {\bf singular radius} $\sradmm$ by
  \begin{equation}
    \sradmm \eqdef \sradmmdef > 0\;,
    \quad\quad\text{and further define}\quad\quad
    \lambda_{\mu,\mu_c} \eqdef \frac{\sradmm}{\rho_{1,0}} = \frac{\mu}{\mu - \mu_c}\;,
  \end{equation}
  as the {\bf induced scaling parameter}. Note that $\rho_{1,0} = 2$ and
  $\lambda_{1,0} = 1$. Further, we define the {\bf parameter rescaling}
  given by
  \begin{equation}
    \widetilde{F}_{\mu,\mu_c} \;\eqdef\; \lambda^{-1}_{\mu,\mu_c}\,F \;=\; \frac{\mu - \mu_c}{\mu}\,F \quad\in \GL^+(n)\;.
  \end{equation}
\end{defi}
For $\mu > 0$ and $\mu_c = 0$, we obtain $\widetilde{F}_{\mu,0} = F$, i.e., the
rescaling is only effective for $\mu_c > 0$. We can now state the key result of
this section which is independent of the dimension.
\begin{lem}[Parameter reduction]
  \label{lem:parameter_reduction}
  Let $n \geq 2$ and let $F \in \GL^+(n)$, then
  \begin{equation}
    \begin{aligned}
      \mu_c \geq \mu > 0 \quad&\Longrightarrow\quad \wmm(\mrot\,;F)
      \;\sim\;
      W_{1,1}(\mrot\,;F)\;,\quad\text{and}\\
      \mu > \mu_c \geq 0 \quad&\Longrightarrow\quad \wmm(\mrot\,;F)
      \;\sim\;
      \wsym(\mrot\,;\widetilde{F}_{\mu,\mu_c})\;.
    \end{aligned}
  \end{equation}
\end{lem}
We have to defer the proof for a bit, since we need some more preparations.

In particular, the foregoing~\leref{lem:parameter_reduction} implies
that we can focus on the classical and non-classical limit cases.
Once these are solved, the solutions for general values of $\mu$
and $\mu_c$ can be recovered. Note that for the non-classical
limit case $(\mu,\mu_c) = (1,0)$, the Cosserat shear-stretch
energy defined in~\eqref{eq:intro:weighted} which is subject to
minimization, takes the explicit form
\begin{equation}
\begin{aligned}
  \wsym \colon \SO(n) \times \GL^+(n)\to \RPosZ\text{,}\quad
  \wsym(\mrot\,;F) = \norm{\sym(\mrot^TF - \id)}^2.
\end{aligned}
\end{equation}

\begin{rem}[Immediate generalizations]
  The parameter reduction in~\leref{lem:parameter_reduction} and the optimality
  results for $n = 2$ in~\secref{sec:minimization} both extend to bijective
  parameter transformations $\tau: \GL^+(n) \to \GL^+(n)$. For example,
  $\tau(F) \eqdef \cof(F)^T$, $\cof(X) \eqdef \det{X}\,X^{-1}$ arises in the
  cofactor shear energy
  \begin{align*}
    W^{\rm cof}_{\mu,\mu_c}(\mrot\,;F) &:\SO(n) \times \GL^+(n) \to \RPosZ\;,\notag\\
    W^{\rm cof}_{\mu,\mu_c}(\mrot\,;F) &\eqdef
    \mu\;\norm{\sym(\cof(\mrot^TF) - \id)}^2 +
    \mu_c\;\norm{\skew(\cof(\mrot^TF) - \id)}^2\\
    &\phantom{:}= \mu\;\norm{\sym(\mrot^T\tau(F) - \id)}^2 +
    \mu_c\;\norm{\skew(\mrot^T\tau(F) - \id)}^2\;.
  \end{align*}
  Our results extend to this case in the natural way, i.e.,
  the optimal rotations are $(\rpolar_{\mu, \mu_c} \circ\, \tau)(F)$.
\end{rem}

In what follows, we take a detailed look at the parameters $\mu$ and
$\mu_c$ and the role they play in the weighted minimization
in~\probref{intro:prob:weighted}.

\subsection{Reduction of the classical parameter range: $\mu_c \geq \mu > 0$}
For the classical parameter range $\mu_c \geq \mu > 0$, we essentially rediscover
Grioli's theorem (see~\theref{theo:intro:grioli}) on the optimality of the polar
factor $\polar(F)$, since the minimization problem reduces to the limit
case $(\mu,\mu_c) = (1,1)$.

\begin{proof}[Proof of~\leref{lem:parameter_reduction} (first part)]
For $\mu_c \geq \mu > 0$, we have $\mu_c - \mu \geq 0$ which gives us the
following lower bound
\begin{align*}
\wmm(\mrot\,;F) \,\eqdef\,& \mu\, \norm{\sym(\mrot^TF) - \id}^2 +
\mu_c\,\norm{\skew(\mrot^TF - \id)}^2\notag\\
\,=\,& \mu\,\norm{\mrot^TF - \id}^2 + (\mu_c - \mu)\,\norm{\skew(\mrot^TF - \id)}^2\\
\,\geq\,& \mu\,\norm{\mrot^TF - \id}^2 = \mu\,W_{1,1}(\mrot\,;F)\,.
\end{align*}
Since $\wmm(\polar(F)\,;F) = \mu\,W_{1,1}(\polar(F)\,;F)$, it
follows that $\wmm \sim W_{1,1}$ for the entire classical parameter
range.\qedhere
\end{proof}

In passing, we have proved the following immediate generalization to
Grioli's theorem.
\begin{cor}
  \label{cor:reduction:grioli}
  Let $\mu_c \geq \mu > 0$ and $F \in \GL^+(n)$, then
\begin{equation*}
  \argmin{\mrot\,\in\,\SO(n)}{\wmm(R\,;F)}
  = \argmin{\mrot\,\in\,\SO(n)}{\left\{\mu\, \norm{\sym(\mrot^TF) - \id}^2 +
\mu_c\,\norm{\skew(\mrot^TF - \id)}^2\right\}}
  = \lbrace{\polar(F)}\rbrace\;\text{,}
\end{equation*}
i.e., the polar factor $\polar(F)$ is the unique global minimizer.\qed
\end{cor}

\subsection{Reduction of the non-classical parameter range: $\mu > \mu_c \geq 0$}

A sparking idea which enters the proof of the second part
of~\leref{lem:parameter_reduction} is due to
M. Hofmann-Kliemt (then at TU Darmstadt~\cite{Hofmann-Kliemt:2007:PHD})
who contributed to the study of the influence of the
parameters $\mu$ and $\mu_c$ by spotting the applicability of the
following elementary identity~\cite{Hofmann-Kliemt:2007:PRL}:
\begin{lem}[Expanding the square]
  \label{lem:tr_complete_square}
  Let $\mrot\,\in\,\SO(n)$ and $F \in \GL^+(n)$, then the following
  identity holds:
  \begin{equation}
    \tr{\left(\mrot^TF - \rho_{\mu,\mu_c}\id\right)^2} = \tr{(\mrot^TF)^2} - 2 \rho_{\mu,\mu_c} \tr{\mrot^TF} + \rho_{\mu,\mu_c}^2\tr{\id}.
  \end{equation}
  \textit{Proof.}
    \vspace{-\baselineskip}
    \begin{align}
      \tr{\left(\mrot^TF - \rho_{\mu,\mu_c}\id\right)^2}
      &= \tr{\left(\mrot^TF - \rho_{\mu,\mu_c}\id\right)\left(\mrot^TF - \rho_{\mu,\mu_c}\id\right)}\notag\\
      &= \tr{(\mrot^TF)^2 - 2\,\rho_{\mu,\mu_c}\;\mrot^TF + \rho_{\mu,\mu_c}^2\,\id}.
    \end{align}
    The claim follows by linearity of the trace operator.\qedhere
\end{lem}

This leads now to a reduction of the minimization of the energy
$\wmm$ to $\wsym$ for the non-classical parameter range.
The main ingredient is the rescaling of the parameter space $\GL^+(n)$
in~\deref{defi:reduction:rescaling}.

\begin{proof}[Proof of~\leref{lem:parameter_reduction} (second part)]
We proceed by successive term expansion, gathering the contributions which
are constant with respect to $\mrot$ at each step. To this end, we split
\begin{equation}
  \wmm(\mrot\,;F) = \underbrace{\mu\; \norm{\sym(\mrot^TF - \id)}^2}_{=:\;\textrm{I}}
  \;+\; \underbrace{\mu_c\;\norm{\skew(\mrot^TF - \id)}^2}_{=:\;\textrm{II}}
\end{equation}
and simplify the summands I and II separately. For the first term, we get
\begin{align}
    \textrm{I} &= \mu\,\norm{\sym{(\mrot^TF - \id)}}^2
      = \frac{\mu}{2} \left( \norm{F}^2 + 2\norm{\id}^2 +
    \tr{(\mrot^TF)^2} -4\,\tr{\mrot^TF} \right)\;.
\intertext{Similarly, for the second term}
  \textrm{II} &= \mu_c\,\norm{\skew(\mrot^TF)}^2
     = \frac{\mu_c}{4}\;\scalprod{\mrot^TF - F^T\mrot}{\mrot^TF -
     F^T\mrot}
     = \frac{\mu_c}{2} \left( \norm{F}^2 - \tr{(\mrot^TF)^2} \right)
\end{align}
is obtained. Summation of I and II while shifting all terms constant in
$\mrot$ to the right yields
\begin{equation}
  \wmm(\mrot\,;F) = \textrm{I} + \textrm{II}
  = \frac{\mu - \mu_c}{2}\,\tr{(\mrot^TF)^2} -
  2\mu\;\tr{\mrot^TF} + \frac{\mu + \mu_c}{2}\,\norm{F}^2 + \mu\,\norm{\id}^2.\label{eq:reduction:wmm_expanded}
\end{equation}
We shall collect all terms which are constant with respect to
$\mrot$ in a sequence
of suitable constants, starting with $c^{(1)}_{\mu,\mu_c}(F) \eqdef \frac{\mu + \mu_c}{2} \norm{F}^2 + \mu\norm{\id}^2$. This yields the expression
\begin{equation}
  \wmm(\mrot\,;F) = \frac{\mu - \mu_c}{2}\,\tr{(\mrot^TF)^2} - 2\mu\;\tr{\mrot^TF} + c^{(1)}_{\mu,\mu_c}(F).
\end{equation}

Introducing the singular radius $\rho_{\mu,\mu_c}$ from~\deref{defi:reduction:lambda}, we can write the preceding equation as follows
\begin{equation}
  \label{WMMFR_ND_CompleteTheSquare_WithSingularRadius}
  \wmm(\mrot\,;F) = \frac{\mu}{\rho_{\mu,\mu_c}}\,\tr{(\mrot^TF)^2}
  - 2\mu\;\tr{\mrot^TF} + c^{(1)}_{\mu,\mu_c}(F)\;,
\end{equation}
which inspires us to define a {\bf rescaled energy}
\begin{equation}
  \label{WMMFRTILDE_ND_Definition}
  \widetilde{W}_{\mu,\mu_c}(\mrot\,;F) \colon \SO(n) \times \GL^+(n)
  \to \RPosZ,\quad \widetilde{W}_{\mu,\mu_c}(\mrot\,;F) \eqdef
  \frac{\rho_{\mu,\mu_c}}{\mu}\;\wmm(\mrot\,;F)\;.
\end{equation}
We now expand using
\eqref{WMMFR_ND_CompleteTheSquare_WithSingularRadius} to get
\begin{align}
  \widetilde{W}_{\mu,\mu_c}(\mrot\,;F)
  &= \frac{\rho_{\mu,\mu_c}}{\mu}\;\wmm(\mrot\,;F)
  = \frac{\rho_{\mu,\mu_c}}{\mu}\; \left(\frac{\mu}{\rho_{\mu,\mu_c}}
    \tr{(\mrot^TF)^2} - 2\mu\;\tr{\mrot^TF} +
    c^{(1)}_{\mu,\mu_c}(F)\right)\notag\\
  &= \tr{(\mrot^TF)^2} -
  2\rho_{\mu,\mu_c}\;\tr{\mrot^TF} + c^{(2)}_{\mu,\mu_c}(F)\;,
\end{align}
with $c^{(2)}_{\mu,\mu_c}(F) \eqdef \frac{\rho_{\mu,\mu_c}}{\mu}\,c^{(1)}_{\mu,\mu_c}(F)$, and observe that $\wmm$ and the rescaled energy $\wmmtilde$ share the
same local and global extrema in $\SO(n)$. This gives us $\wmmtilde \sim \wmm$
which also holds for the particular choice of parameters $\mu = 1$ and
$\mu_c = 0$, i.e., $\widetilde{W}_{1,0} \sim \wsym$. For the latter specific
choice of parameters, the rescaled energy takes the form
\begin{equation}
\widetilde{W}_{1,0}(\mrot\,;F) = \tr{(\mrot^TF)^2} -
              2\rho_{1, 0}\,\tr{\mrot^TF} +
              c^{(2)}_{1,0}(F).
\end{equation}
The next step of the proof is to show an affine relation between
$\wmmtilde$ and $\widetilde{W}_{1,0}$. With~\leref{lem:tr_complete_square},
we proceed by completing the square to get
\begin{align}
  \widetilde{W}_{\mu,\mu_c}(\mrot\,;F) &= \tr{(\mrot^TF)^2}
                  - 2 \rho_{\mu,\mu_c}\,\tr{\mrot^TF}
                  + c^{(2)}_{\mu,\mu_c}(F)\notag \\
              &= \tr{(\mrot^TF)^2}
                  - 2 \rho_{\mu,\mu_c}\,\tr{\mrot^TF} +
                 \rho_{\mu,\mu_c}^2 - \rho_{\mu,\mu_c}^2
                 + c^{(2)}_{\mu,\mu_c}(F)\notag \\
              &= \tr{(\mrot^TF - \rho_{\mu,\mu_c}\id)^2}
                 + c^{(3)}_{\mu,\mu_c}(F)\;,
\end{align}
where $c^{(3)}_{\mu,\mu_c}(F) \eqdef c^{(2)}_{\mu,\mu_c}(F) - \rho_{\mu,\mu_c}^2$.
Inserting $\mu = 1$ and $\mu_c = 0$, we obtain the special case
\begin{align}
\label{WSYMTILDE_ND_COMPLETED_SQUARE}
  \widetilde{W}_{1,0}(\mrot\,;F)
  = \tr{(\mrot^TF - \rho_{1,0}\id)^2} + c^{(3)}_{1,0}(F)\;.
\end{align}

We can now reveal the connection between the minimization problem
with parameters $\mu > \mu_c \geq 0$ and the non-classical limit
case $(\mu,\mu_c) = (1,0)$. Note first that
\begin{align}
  \widetilde{W}_{\mu,\mu_c}(\mrot\,;F) &= \tr{\left(\mrot^TF -
    \rho_{\mu,\mu_c}\id\right)^2} +
  c^{(3)}_{\mu,\mu_c}(F)
  = \tr{\left(\mrot^TF -
    \frac{\rho_{\mu,\mu_c}}{\rho_{1,0}}\;\rho_{1,0}\;\id\right)^2} +
  c^{(3)}_{\mu,\mu_c}(F)\notag\\
  &= \lambda_{\mu,\mu_c}^2\,\tr{\left(\mrot^T\widetilde{F}_{\mu,\mu_c} - \rho_{1,0}\;\id\right)^2} + c^{(3)}_{\mu,\mu_c}(F)\;.\label{eq:reduction:wmmtildeconnection}
\end{align}
In the last equation, we easily discover the trace term of equation
\eqref{WSYMTILDE_ND_COMPLETED_SQUARE} with one essential change: $F$ was replaced
by $\widetilde{F}_{\mu,\mu_c} \eqdef \lambda^{-1}_{\mu,\mu_c}F$.
We now solve~\eqref{WSYMTILDE_ND_COMPLETED_SQUARE} for the trace term,
and insert $\widetilde{F}_{\mu,\mu_c}$ for the parameter $F$. This gives
\begin{equation}
  \tr{(\mrot^T\widetilde{F}_{\mu,\mu_c} - \rho_{1,0}\id)^2}
  = \widetilde{W}_{1,0}(\mrot\,;\widetilde{F}_{\mu,\mu_c}) - c^{(3)}_{1,0}(\widetilde{F}_{\mu,\mu_c})\;.
\end{equation}
Next, we substitute the trace term in~\eqref{eq:reduction:wmmtildeconnection}
by its expression in terms of $\wsymtilde$ and finally obtain
\begin{align}
  \label{WMM_TO_WSYM_TRAFO}
  \widetilde{W}_{\mu,\mu_c}(\mrot\,;F)
  &= \lambda_{\mu,\mu_c}^2\;\left(\wsymtilde(\mrot\,;\widetilde{F}_{\mu,\mu_c})
    - c^{(3)}_{1,0}(\widetilde{F}_{\mu,\mu_c})\right) + c^{(3)}_{\mu,\mu_c}(F)\notag\\
  &= \lambda_{\mu,\mu_c}^2\;\wsymtilde(\mrot\,;\widetilde{F}_{\mu,\mu_c})
  + c^{(4)}_{\mu,\mu_c}(F)\;,
\end{align}
with $c^{(4)}_{\mu,\mu_c}(F) \eqdef c^{(3)}_{\mu,\mu_c}(F) - \lambda_{\mu,\mu_c}^2\,c^{(3)}_{1,0}(\widetilde{F}_{\mu,\mu_c})$. This establishes the missing link
$\wmmtilde(\mrot\,;F) \sim \wsymtilde(\mrot\,;\widetilde{F}_{\mu,\mu_c})$,
since $\lambda_{\mu,\mu_c}^2 > 0$ and the final constant $c^{(4)}_{\mu,\mu_c}(F)$
depends only on $F$. With this, the chain
\begin{equation}
\wmm(\mrot\,;F) \sim
\wmmtilde(\mrot\,;F) \sim
\wsymtilde(\mrot\,;\widetilde{F}_{\mu,\mu_c}) \sim
\wsym(\mrot\,;\widetilde{F}_{\mu,\mu_c})
\end{equation}
is now complete. All four energies give rise to the same
energy-minimizing rotations in $\SO(n)$.\qedhere
\end{proof}%

Once the optimal energy-minimizing rotations for $\wsym$ are available,
the optimal rotations for $\wmm$ with general weights $\mu > \mu_c \geq 0$
can be directly inferred by a substitution of $F$ with
$\widetilde{F}_{\mu,\mu_c}$. This procedure is detailed
in~\secref{subsec:minimization:reconstruction}. In this
sense, surprisingly, the non-classical case $\mu > \mu_c > 0$, i.e.,
with \emph{strictly positive} Cosserat couple modulus $\mu_c$,
is completely governed by the case with zero Cosserat couple
modulus $\mu_c = 0$ which is highly interesting in view
of~\cite{Neff_ZAMM05}!

%
%
%
%

%
%
%
%

%
%
%
%

%
%
%
%
%
%
%
%
%
%

%
%
 %
\countres
\makeatletter{}%
\section{Optimal rotations for the Cosserat shear-stretch energy}
\label{sec:minimization}
In this section, we compute explicit representations of optimal
planar rotations for the Cosserat shear-stretch energy, i.e.,
we focus on dimension $n = 2$. The parameter reduction strategy
in~\leref{lem:parameter_reduction} allows us to concentrate
our efforts towards the construction of explicit solutions
to~\probref{intro:prob:planar} on two representative pairs of parameter
values $\mu$ and $\mu_c$. The classical
regime is characterized by the limit case $(\mu,\mu_c) = (1,1)$ and the unique
minimizer is given by the polar factor $\polar(F)$ for any dimension
$n \geq 2$, see~\coref{cor:reduction:grioli}. The non-classical case represented
by $(\mu,\mu_c) = (1,0)$ turns out to be much more interesting and we compute
all global non-classical minimizers $\rpolar_{1,0}(F)$ for $n = 2$. This is the
main contribution of this section. Furthermore, we derive the associated
reduced energy levels $W^{\rm red}_{1,1}(F)$ and $W^{\rm red}_{1,0}(F)$
which are realized by the corresponding optimal Cosserat microrotations.
Finally, we reconstruct the minimizing rotation angles for general values
of $\mu$ and $\mu_c$ from the classical and non-classical limit cases.

\subsection{Explicit solution for the classical parameter range: $\mu_c \geq \mu > 0$}
By~\coref{cor:reduction:grioli} the polar factor $\polar(F)$ is uniquely
optimal for the classical parameter range in any dimension $n \geq 2$.
Let us give an explicit representation for $n = 2$ in terms of
$\palpha \in (-\pi,\pi]$. In view of the parameter reduction, distilled
in~\leref{lem:parameter_reduction}, it suffices to compute the set of
optimal rotation angles for the representative limit case
$(\mu,\mu_c) = (1,1)$.

Thus, to obtain an explicit representation of $\palpha \in (-\pi,\pi]$ which
characterizes the polar factor $\polar(F)$ in dimension $n = 2$, we
consider
\begin{equation}
  \argmin{\malpha\;\in\;[-\pi,\pi]} W_{1,1}(R(\malpha)\,;F)
  = \argmin{\malpha\;\in\;[-\pi,\pi]}{\hsnorm{\left[\begin{pmatrix}
      \cos \malpha & -\sin \malpha\\
      \sin \malpha & \cos \malpha
\end{pmatrix}^T
\begin{pmatrix}
    F_{11} & F_{12}\\
    F_{21} & F_{22}\\
\end{pmatrix}
-\begin{pmatrix}
 1 & 0\\
 0 & 1
\end{pmatrix}\right]}^2}\,.
\end{equation}
Let us introduce the rotation
$J \eqdef \begin{pmatrix} 0 & -1\\ 1 & 0 \end{pmatrix} \in \SO(2)$.
Its application to a vector $v \in \R^2$ corresponds to multiplication
with the imaginary unit $i \in \C$. In what follows, the quantities
$\tr{F} = F_{11} + F_{22}$ and $\tr{JF} = - F_{21} + F_{12}$
play a particular role and we note the identity
\begin{equation}
\tr{F}^2 + \tr{JF}^2 = \hsnorm{F}^2 + 2\,\det{F} = \tr{U}^2\;.
\end{equation}

The reduced energy $W_{1,1}^{\rm red}(F) \eqdef \min_{R\in\SO(n)} W_{1,1}(R\,;F)$
realized by the polar factor $\polar(F)$ can be shown to be the euclidean
distance of an arbitrary $F$ in $\R^{n\times n}$ to $\SO(n)$. For $n = 2$, we
obtain

\begin{theo}[Euclidean distance to planar rotations]
  Let $F \in \GL^+(2)$, then
  \begin{equation}
    W_{1,1}^{\rm red}(F) = \dist^2(F,\SO(2)) = \hsnorm{U - \id}^2
    = \hsnorm{F}^2 - 2\,\sqrt{\hsnorm{F}^2 + 2\,\det{F}} + 2\;.
  \end{equation}
  The unique optimal rotation angle realizing this minimial
  energy level satisfies the equation
  \begin{equation}
    \begin{pmatrix}
      \sin\palpha\\
      \cos\palpha
    \end{pmatrix}
    =\frac{1}{\tr{U}}
    \begin{pmatrix}
      -\tr{JF} \\
      \tr{F}
    \end{pmatrix}\;.
  \end{equation}
  In particular, we have
  $\palpha(F) = \arccos\left(\frac{\tr{F}}{\tr{U}}\right)$.
\end{theo}
\begin{proof}
  See~\cite{Martin:2015:RCPC}[Appendix A2.1].\qedhere
\end{proof}

\begin{cor}[Explicit formula for $\polar(F)$]
  \label{cor:polar_planar}
  Let $F \in \GL^+(2)$, then the polar factor $\polar(F)$ has the
  explicit representation
  \begin{equation}
    \polar(F) =
    R(\palpha) \eqdef \begin{pmatrix}
      \cos\palpha & -\sin\palpha\\
      \sin\palpha &  \cos\palpha
    \end{pmatrix}
    = \frac{1}{\tr{U}}
    \begin{pmatrix}
      \phantom{-}\tr{F}  & \tr{JF}\\
      -\tr{JF} & \tr{F}
    \end{pmatrix}\;.
  \end{equation}
\end{cor}

\subsection{Symmetry of the first Cosserat deformation tensor}
\label{subsec:symmetry}
The goal of this subsection is two-fold: first, we want to solve the
equation \mbox{$\skew(\mrot^TF) = 0$} for $\mrot \in \SO(2)$ which is equivalent
to $\mrot^TF \in \Sym(2)$. The unique polar factor $\polar(F)$ is
certainly a solution, but are there others? Second, we want to introduce an
approach based on a rotation $\hat{R}$ \emph{relative to the polar
  factor} $\polar(F)$. This turns out to be essential to fully grasp the
symmetry of the non-classical minimizers $\rpolar(F)$.
This is the content of the next lemma.\footnote{Cf. also~\cite[Eq. (2.12)]{Neff_Biot07} for the case $n = 3$.}
\begin{lem}[Symmetry of the planar first Cosserat deformation tensor]
  \label{lem:mstretch_symmetry}
  Let $F \in \GL^+(2)$ be given. The first Cosserat deformation tensor
  $\mstretch(\mrot) \eqdef \mrot^TF$ is symmetric if and only if
  \begin{equation}
    \mrot = \pm\polar(F)\,.
  \end{equation}
\end{lem}
\begin{proof}
  Let us first transform the equation into the orthogonal
  coordinate system induced by the principal directions of stretch.
  The orthogonal basis given by the eigendirections of $U$ makes up the
  columns of a matrix $Q \in \SO(2)$. As it turns out, it is natural
  to define a relative rotation in principal stretch coordinates which
  is given by
  \begin{equation}
    \label{eq:hat_Ralpha_planar}
    \hat{R}(\beta) \eqdef Q^T\mrot(\alpha)^T\polar(F)Q
    = \mrot(\alpha)^T\polar(F)\;.
  \end{equation}
  Note that the rightmost equality holds only for $\SO(2)$, because it is
  commutative. We expand
  $\mstretch \eqdef \mrot^TF = \mrot^T\polar(F)U = \mrot^T\polar(F)QDQ^T$
  and exploit $Q^T\skew(X)Q = \skew(Q^TXQ)$, i.e., the fact that $\skew$ is
  an isotropic tensor function. This gives
  \begin{equation}
    \skew(\mstretch) = 0 \isequivto Q^T\skew(\mrot^T\polar(F)QDQ^T)Q = 0
    \isequivto \skew(\underbrace{Q^T\mrot^T\polar(F)Q}_{=: \hat{R}}D\id) = 0\;,
  \end{equation}
  where $D \eqdef \diag(\sigma_1, \sigma_2) \eqdef \begin{pmatrix}\sigma_1 & 0\\
    0 & \sigma_2\end{pmatrix}$, is the diagonalization of $U$. A simple
    computation in components leads to the necessary and sufficient condition
  \begin{equation}
    0 = \skew(\hat{R}(\beta)D) = \begin{pmatrix}
      0 & \frac{1}{2}(\sigma_1 + \sigma_2)\sin\beta\\
      -\frac{1}{2}(\sigma_1 + \sigma_2)\sin\beta & 0\\
    \end{pmatrix}
    = \sin\beta\frac{\tr{D}}{2}
    \begin{pmatrix}
      0 & 1\\
      -1 & 0
    \end{pmatrix}\;.
  \end{equation}
  We conclude that the necessary and sufficient condition
  for $\mstretch \in \Sym(2)$ is $\sin(\beta) = 0$. Let us
  restrict $\beta \in (-\pi,\pi]$, then $\beta = 0 \lor \beta = \pi$,
    i.e., $\hat{R}(\beta) = \pm\id$. Substituting this
    into~\eqref{eq:hat_Ralpha_planar} and solving for $R(\alpha)$
    yields the claim.\qedhere
\end{proof}%

\begin{rem}[Symmetry of strains vs. symmetry of stresses]
  Consider an energy $W^\sharp(\mstretch)$. Then, the symmetry of the
  Cauchy stress tensor $\sigma(F) \eqdef 1/\det{F}\,R\,\mathrm{D}_{\mstretch} W^\sharp(\mstretch)\,F^T$ is equivalent to the symmetry of
  $\mathrm{D}_{\mstretch} W^\sharp(\mstretch)\,\mstretch^T$ as was shown
  in~\cite{Neff_Biot07}. We recall that the microstrain tensor
  $\mstretch - \id_2$ is symmetric if and only if
  $R = \pm\polar(F)$. This symmetry does
  imply that the Cauchy stress tensor is symmetric. \textbf{It is,
    however, possible that non-symmetric microstrains induce a
    symmetric Cauchy stress tensor.} This may be unexpected,
  but it is the natural scenario for the case of
  non-classical optimal Cosserat rotations. A thorough discussion
  is given in~\cite{Neff_Biot07}.
\end{rem}

\subsection{The limit case $(\mu,\mu_c) = (1,0)$ for $\mu > \mu_c \geq 0$}
We now approach the more interesting non-classical limit case
$(\mu, \mu_c) = (1,0)$ and compute the optimal rotations for
$W_{\mu,\mu_c}(\mrot\,;F)$. Note that, due to~\leref{lem:parameter_reduction},
this limit case represents the entire non-classical parameter
range $\mu > \mu_c \geq 0$.

In the proof to~\leref{lem:parameter_reduction}, we have introduced
$\widetilde{W}_{1,0} \sim W_{1,0}$, i.e., a modified energy that gives
rise to the same optimal rotations. In a similar spirit, inserting the
particular values $(\mu,\mu_c) = (1,0)$ into~\eqref{eq:reduction:wmm_expanded}
from the proof of~\leref{lem:parameter_reduction}, we can specialize to
$n = 2$ as follows:
\begin{align}
  W_{1,0}(\mrot\,;F)
  &\stackrel{\phantom{\eqref{eq:square_trace_id}}}{=}\;\frac{1}{2}\,\tr{(\mrot^TF)^2} - 2\,\tr{\mrot^TF} + \frac{1}{2}\,\hsnorm{F}^2 + \hsnorm{\id_2}^2\notag\\
  &\stackrel{\eqref{eq:square_trace_id}}{=}\;
    \underbrace{\frac{1}{2}\,\tr{\mrot^TF}^2 - 2\,\tr{\mrot^TF}}_{=:\;\mathring{W}(\mrot\,;F)}
    + \underbrace{\frac{1}{2}\,\norm{F}^2 - \det{F} + 2}_{=:\;\mathring{c}(F)}
    =: \mathring{W}(\mrot\,;F) + \mathring{c}(F)\;.\label{eq:defi:wring}
\end{align}
This implies $\mathring{W} \sim W_{1,0}$, since both energies
differ by a constant (with respect to $R$)
\begin{equation}
  \mathring{c}(F) \; = \; \frac{1}{2}\norm{F}^2 - \det{F} + 2
                  \;\stackrel{\eqref{eq:square_trace_id}}{=} \;\frac{1}{2}\tr{U}^2 -2\,\det{U} + 2\;.
\end{equation}

The next step is to determine the critical rotations for $W_{1,0}(\mrot\,;F)$
by taking derivatives w.r.t. $\mrot\,\in\,\SO(2)$. Let us compute
the necessary conditions.

\begin{theo}[Characterization of the critical rotations for $W_{1,0}(\mrot\,;F)$]
  \label{theo:char_rcrit10}
  Let $F \in \GL^+(2)$. A rotation $\mrot\,\in\,\SO(2)$ is a critical
  point for the energy $W_{1,0}(\mrot\,;F)$ if and only if
  \begin{equation*}
    \skew(\mrot^TF) = 0\;,
    \quad\quad \text{or}\quad\quad
    \tr{\mrot^TF}  = 2 \;\land\; \tr{U} \geq 2\;.
  \end{equation*}
\end{theo}
\begin{proof}
  Taking variations $\delta \mrot = A\cdot\mrot\, ,A\in\so(2)$,
  we arrive at the stationarity condition
  \begin{align}
    \forall\, A\in\so(2):\quad \left(\tr{\mrot^TF}-2\right)\, \scalprod{\mrot^TF}{A}=0\;.
  \end{align}
  This equation holds good, if and only if either of the two factors
  on the left hand side vanishes.

  In~\leref{lem:mstretch_symmetry}, we have shown that $\skew(\mrot^TF) = 0$,
if and only if $R = \pm \polar(F)$. Let us discuss the second possibility
\begin{align}
  \tr{\mrot(\malpha)^TF}=2
  \quad \isequivto \quad
  \scalprod{\underbrace{\begin{pmatrix}
      \cos\malpha\\
      \sin\malpha
  \end{pmatrix}}_{=: v(\malpha)}}{\underbrace{\begin{pmatrix}
      \tr{F} \\ \tr{JF}
  \end{pmatrix}}_{=: w(F)}} = 2
\end{align}
as an equation for $\malpha$. The equation on the right hand side
is easily obtained by a short computation in components. From the relation
$\scalprod{v(\malpha)}{w(F)} = \cos\malpha\,\norm{v(\malpha)}\norm{w(F)} = 2$
it follows that the angle $\malpha$ can only be solved for, if
\begin{align}
\frac{2}{\norm{w(F)}} = \frac{2}{\sqrt{\tr{F}^2+\tr{JF}^2}} = \frac{2}{\tr{U}} \le 1\;.
\end{align}
For $\tr{U} \geq 2$, due to the symmetry of the cosine, there exist two symmetric
solutions for $\malpha$ (which may coincide). For $0 < \tr{U} < 2$ there is
no solution.\qedhere
\end{proof}

Let us enumerate the previously obtained necessary conditions
for a critical point:
\begin{equation*}
  \mathrm{1}.)\quad \mrot=-\polar(F)\;, \quad\quad
  \mathrm{2}.)\quad \mrot=+\polar(F)\;, \quad\text{and}\quad
  \mathrm{3}.)\quad \tr{\mrot^TF} = 2 \;\land\; \tr{U} \geq 2\;.
\end{equation*}
Note that for the two classical critical points $\mp\polar(F)$, we have
$\tr{\mrot^TF} = \mp\tr{U}$. Due to the particular expression of the energy
$\mathring{W}$ in terms of $\tr{\,\mstretch\,}$, we can insert the critical
values of $\tr{\mrot^TF}$ into the defining
equation~\eqref{eq:defi:wring} for $\mathring{W}$ to obtain the associated
critical energy levels:
\begin{equation*}
\label{eq:defi_Wi}
\mathring{W}^{(1)}(F) = \frac{1}{2}\tr{U}^2 + 2\,\tr{U}\;,\quad\quad
\mathring{W}^{(2)}(F) = \frac{1}{2}\tr{U}^2 - 2\,\tr{U}\;,\quad\text{and}\quad
\mathring{W}^{(3)}(F) = -2\;.
\end{equation*}
Our first observation is that $\mathring{W}^{(1)}(F) \geq \mathring{W}^{(2)}(F)$
for all $F \in \GL^+(2)$. Further, if $\tr{U} \ge 2$, then the branch
$\mathring{W}^{(3)}(F)$ exists and we have $\mathring{W}^{(1)}(F) \geq \mathring{W}^{(2)}(F) \geq \mathring{W}^{(3)}(F)$, i.e., the non-classical branch
$\mathring{W}^{(3)}$ realizes the global minimum when it exists.
For $0 < \tr{U} < 2$, the branch $\mathring{W}^{(3)}$ does not exist and the
global minimum is realized by the classical branch $\mathring{W}^{(2)}$,
i.e., $\polar(F)$ is uniquely optimal. Finally, for $\tr{U} = 2$,
we have $\mathring{W}^{(1)} > \mathring{W}^{(2)} = \mathring{W}^{(3)}$.

\begin{theo}[The formally reduced energy $W^{\rm red}_{1,0}(F)$]
  \label{theo:wred10}
  Let $F \in \GL^+(2)$, let the energies $\mathring{W}^{(1)}$, $\mathring{W}^{(2)}$
  and $\mathring{W}^{(3)}$ as above and $W^{(i)}(F) \eqdef \mathring{W}^{(i)}(F) +  \mathring{c}(F), i = 1,2,3$. Then, the {\bf formally reduced energy}
  \begin{equation}
    W_{1,0}^{\rm red}(F) \eqdef \min_{R \; \in \; \SO(2)}W_{1,0}(\mrot\,;F)
                       \eqdef \min_{\mrot\,\in\,\SO(2)} \norm{\sym(\mrot^TF-\id)}^2
  \end{equation}
is given by
\begin{equation}
  \label{eq:minimization:wsymred}
    W_{1,0}^{\rm red}(F)
    = \begin{cases}
      W^{(2)}(F) = \tr{(U - \id)^2} = \dist^2(F,\SO(2))\;,   &\quad\text{if} \quad \tr{U} < 2\\
      W^{(3)}(F) = \frac{1}{2}\hsnorm{F}^2 - \det{F}
                \stackrel{\eqref{eq:square_trace_id}}{=} \frac{1}{2}\,\tr{U}^{2} - 2\,\det{U}\;, &\quad\text{if}\quad \tr{U} \ge 2\;.
    \end{cases}
  \end{equation}
\end{theo}
\begin{proof}
  It suffices to add the constant $\mathring{c}(F)$ to the minimal
    energy levels for $\mathring{W}(\mrot\,;F)$.
  Since $W^{(2)}(F)$ corresponds to $R = \polar(F)$ for which
  $R^TF = U$ is symmetric, we find that
  $W^{(2)}(F) = \min_{R \in \SO(2)} W_{1,1}(R\,;F) = \dist^2(F,\SO(2))$.
  Note that for $\tr{U} = 2$, we have $W^{(2)}(F) = W^{(3)}(F)$.\qedhere
\end{proof}%

It is well-known that any orthogonally invariant energy density
$W(F)$ admits a representation in terms of the singular values of $F$, i.e.,
in the eigenvalues of $U$. Let us give this representation.
\begin{cor}[Representation of $W^{\rm red}_{1,0}(F)$ in the singular values
    of $F$]
  \label{cor:wred10_sing}
  Let $F \in \GL^+(2)$ and denote its singular values by $\sigma_i$, $i = 1,2$.
  The representation of $W^{\rm red}_{1,0}(F)$ in the singular values of $F$
  is given by
  \begin{equation}
    W_{1,0}^{\rm red}(F) = W_{1,0}^{\rm red}(\sigma_1,\sigma_2) =
    \begin{cases}
      (\sigma_1 - 1)^2 + (\sigma_2 - 1)^2\;,  &\text{if}\quad \sigma_1 + \sigma_2  < 2\\
      \frac{1}{2}(\sigma_1 - \sigma_2)^2\;,  &\text{if}\quad\sigma_1 + \sigma_2 \geq 2\;.
    \end{cases}
  \end{equation}\
\end{cor}
\begin{proof}
  We insert $\norm{F}^2 = \norm{U}^2 = \sigma_1^2 + \sigma_2^2$ and
  $\det{F} = \det{U} = \sigma_1\sigma_2$
  into~\eqref{eq:minimization:wsymred}. It is not hard to see
  that both pieces of the energy coincide for $\sigma_1 + \sigma_2 = 2$.\qedhere
\end{proof}%
Note that the previous formulae are independent of the enumeration
of the singular values.

\subsubsection{Optimal relative rotations for \texorpdfstring{$\mu = 1$}{mu = 1} and \texorpdfstring{$\mu_c = 0$}{muc = 0}}

Our next goal is to compute explicit representations of the rotations
$\rpolar^\pm_{1,0}(F)$ which realize the minimal energy level $W^{(3)}(F)$
in the non-classical limit case $(\mu,\mu_c) = (1,0)$. This is the content
of the next theorem for which we now prepare the stage with the following
\begin{lem}
  \label{lem:eq_tr_hatrd_eq_2}
  Let $D = \diag(\sigma_1,\sigma_2) > 0$, i.e, a diagonal matrix with
  strictly positive diagonal entries. Then, assuming $\tr{D} \geq 2$,
  the equation $\tr{R(\beta)\,D} = 2$ has the following solutions
    \begin{equation}
      \beta^\pm
      \;=\;
      \pm\,\arccos\left(\frac{2}{\tr{D}}\right) \quad\in [-\pi,\pi]\;.
    \end{equation}
    For $\tr{D} < 2$, there exists no solution, but we can define
    $\beta = \beta^\pm \eqdef 0$ by continuous extension.
\end{lem}
\begin{proof}
  We compute
  \begin{equation}
    2 = \tr{\begin{pmatrix}
        \cos \beta  & -\sin \beta\\
        \sin \beta  &  \cos \beta
      \end{pmatrix}
      \begin{pmatrix}
        \sigma_1 & 0\\
        0 & \sigma_2
    \end{pmatrix}}
    = (\sigma_1 + \sigma_2)\cos\beta
    = \tr{D}\cos\beta\;.
  \end{equation}
  Since $\tr{D} > 0$, we may divide to obtain the relation
  $\cos \beta = 2/\tr{D}$. This is solvable if and only
  if $2/\tr{D} \leq 1$ which is equivalent to $\tr{D} \geq 2$.
  There are two symmetric solutions $\beta^\pm = \pm\,\arccos(2/\tr{D})$.
  Since both vanish for $\tr{D} = 2$, we can
  continously extend $\beta = \beta^\pm \eqdef 0$ for $\tr{D} < 2$.\qedhere
\end{proof}%

Our~\figref{fig:alpha_rel_plot} shows a plot of the optimal
relative rotation angle $\beta(\tr{U})$.
In the classical parameter range $0 < \tr{U} \leq 2$,
$\palpha(F)$ is uniquely optimal and $\beta$ vanishes identically.
In \mbox{$\tr{U} = 2$}, a classical pitchfork bifurcation occurs. In
particular, due to $\tr{U(\id_2)} = \tr{\id_2} = 2$,
the identity matrix is a bifurcation point of $\beta^\pm(F)$.
Further, we note that the branches $\beta^\pm(\tr{U}) = \pm\arccos(2/\tr{U})$
\emph{are not differentiable} at $\tr{U} = 2$. This has implications
on the interaction of the Cosserat shear-stretch energy with the
Cosserat curvature energy $W_{\rm curv}$.
\begin{figure}
  \begin{center}
    \includegraphics{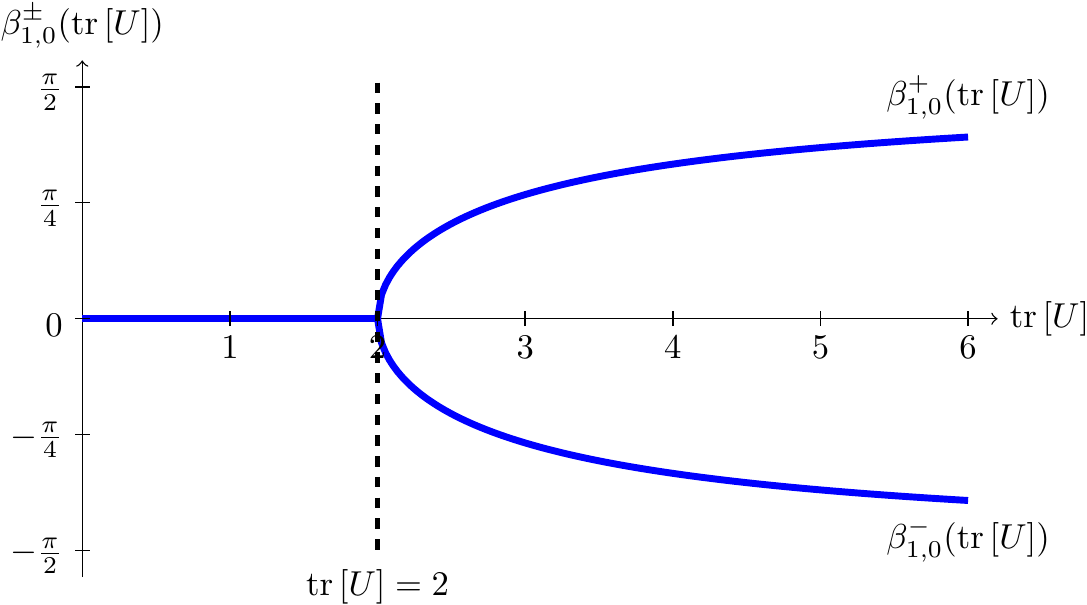}
  \end{center}
  \caption[The optimal relative rotation angle $\beta_{1,0}$ as a function of $\tr{U}$.]{
    Plot of the two optimal relative rotation angles
    $\beta_{1,0}^\pm = \pm\arccos\left(\frac{2}{\tr{U}}\right)$
    for the non-classical limit case $(\mu,\mu_c) = (1,0)$.
    Note the pitchfork bifurcation in $\tr{U} = \rho_{1,0} = 2$.
    For $0 < \tr{U} < 2$, the polar angle $\palpha$ is uniquely
    optimal and the relative rotation angle $\beta$ vanishes
    identically.\label{fig:alpha_rel_plot}%
    }
\end{figure}

\begin{theo}[Optimal non-classical microrotation angles $\malpha_{1,0}^\pm$]
  \label{theo:malpha10pm}
  Let $F \in \GL^+(2)$ and consider $(\mu,\mu_c) = (1,0)$. The optimal
  rotation angles for $\wsym$ are given by
  \begin{equation}
    \malpha_{1,0}^{\pm}(F) =
    \begin{cases}
      \palpha(F) = \arccos(\frac{\tr{F}}{\tr{U}})
      & ,\quad\text{if}\qquad \tr{U} < 2\\
      \palpha(F) \pm\,\arccos\left(\frac{2}{\tr{U}}\right)
      & ,\quad\text{if}\qquad \tr{U} \geq 2\;.
    \end{cases}
  \end{equation}
\end{theo}
\begin{proof}
  The first statement follows from \theref{theo:wred10} which shows
  that $\polar(F)$ realizes $W^{(2)}$. Further, by~\theref{theo:char_rcrit10}
  (see also the proof) the branch $W^{(2)}$ uniquely corresponds to
  $\polar(F)$. In other words
  \begin{equation}
    \tr{U} < 2
    \quad\Longrightarrow\quad
    \malpha_{\mu,\mu_c}(F) = \palpha(F)
    \quad\Longrightarrow\quad
    \rpolar(F) = \polar(F)\;.
  \end{equation}
  Let us now assume $\tr{U} \geq 2$. In this case,
  by~\theref{theo:char_rcrit10}, globally
  energy minimizing rotations $\mrot$ realize $W^{(3)}$.
  Thus, $\malpha \in (-\pi,\pi]$ is a solution of
    $\tr{\,\mrot(\malpha)^TF} = 2$ for given $F \in \GL^+(2)$.
    Consider again the relative rotation (cf. the proof
    of~\leref{lem:mstretch_symmetry}) given by
    \begin{equation}
      \hat{R}(\beta) \eqdef \mrot(\malpha)^T\polar(F)\;.
      \label{eq:hat_Ralpha_rel_w10}
  \end{equation}
  The uniqueness of $\polar(F)$ implies a one-to-one correspondence between
  $\mrot(\malpha)$ and $\hat{R}(\beta)$. In terms of rotation angles, we
  find that
  \begin{equation}
    \malpha = \palpha - \beta\label{eq:angle_ids}\;,
  \end{equation}
  where $\palpha$ denotes the rotation angle of the
  polar rotation $\polar(F)$, i.e., $R(\palpha) = \polar(F)$.
  After a transformation into the coordinate system given by the principal
  directions of stretch (i.e., given by the eigendirections of $U$), we obtain
  \begin{align}
    2 &= \tr{\mrot^TF} = \tr{\hat{R}(\beta)D}\;.
  \end{align}
  Applying~\leref{lem:eq_tr_hatrd_eq_2} we find that there are two
  energy-minimizing relative rotation angles
  \begin{equation}
    \beta^\pm = \pm\,\arccos\left(\frac{2}{\tr{D}}\right) = \pm\,\arccos\left(\frac{2}{\tr{U}}\right)\quad,\quad\text{for}\;\tr{U} \geq 2\;.
  \end{equation}
  We can now solve~\eqref{eq:angle_ids} for the corresponding microrotation
  angles $\malpha_{1,0}^\pm$ which gives
  \begin{equation}
    \malpha^\pm_{1,0}  = \palpha - \beta^\pm
          = \palpha \mp\,\arccos\left(\frac{2}{\tr{U}}\right)\;.
  \end{equation}
  The second equality is just another application of~\leref{lem:trU}.\qedhere
\end{proof}%

\subsection{General values for~\texorpdfstring{$\mu$}{mu} and \texorpdfstring{$\mu_c$}{muc}}
\label{subsec:minimization:reconstruction}
The reduction for $\mu$ and $\mu_c$ in~\leref{lem:parameter_reduction} asserts
that the optimal rotations for arbitrary values of $\mu > 0$ and $\mu_c \geq 0$
can be reconstructed from the limit cases $(\mu, \mu_c) = (1,1)$ and
$(\mu, \mu_c) = (1,0)$. We now detail this procedure which essentially
exploits~\deref{defi:reduction:rescaling}.

Note first that the rescaled deformation gradient $\widetilde{F}_{\mu,\mu_c} \eqdef \lambda^{-1}_{\mu,\mu_c} F$ induces a rescaled stretch tensor
\begin{equation}
  \widetilde{U}_{\mu,\mu_c} = \sqrt{(\widetilde{F}_{\mu,\mu_c})^T\widetilde{F}_{\mu,\mu_c}} = \lambda^{-1}_{\mu,\mu_c}\cdot U\;.
\end{equation}
The right polar decomposition takes the form
$\widetilde{F}_{\mu,\mu_c} = \polar(\widetilde{F}_{\mu,\mu_c})\,\widetilde{U}_{\mu,\mu_c}$.
From $\polar(\widetilde{F}_{\mu,\mu_c}) = \widetilde{F}_{\mu,\mu_c}\widetilde{U}_{\mu,\mu_c}^{-1}$ follows the scaling invariance $\polar(\widetilde{F}_{\mu,\mu_c}) = \polar(F)$.
For the non-classical parameter range $\mu > \mu_c \geq 0$, the quantity
\begin{equation}
  \tr{\widetilde{U}_{\mu,\mu_c}} = \tr{ \lambda_{\mu,\mu_c}^{-1} \cdot U} = \frac{\rho_{1,0}}{\rho_{\mu_,\mu_c}}\tr{U}
\end{equation}
plays an essential role. This leads us to
\begin{equation}
  \tr{\widetilde{U}_{\mu,\mu_c}} \geq 2 = \rho_{1,0}
  \quad \isequivto \quad
  \tr{\frac{\rho_{1,0}}{\rho_{\mu_,\mu_c}}\cdot U} \geq \rho_{1,0}
  \quad \isequivto \quad
  \tr{U} \geq \rho_{\mu,\mu_c}\;.
\end{equation}
In particular, this implies that the bifurcation in $\tr{U}$
allowing for non-classical optimal planar rotations is
characterized by the singular radius $\sradmm \eqdef \sradmmdef$.

\begin{theo}
  Let $F \in \GL^+(2)$. For $\mu_c \geq \mu > 0$
  the optimal microrotation angle is given by
  \begin{equation}
    \alpha_{\mu,\mu_c}(F)
    = \palpha(\widetilde{F}_{\mu,\mu_c})
    = \palpha(F) = \arccos\left(\frac{\tr{F}}{\tr{U}}\right)\;.
  \end{equation}
  For $\mu > \mu_c \geq 0$, the two optimal rotation angles
  are given by
  \begin{equation}
    \malpha^{\pm}_{\mu,\mu_c}(F) = \malpha^{\pm}_{1,0}(\widetilde{F}_{\mu,\mu_c})
  = \begin{cases}
    \palpha(F) = \arccos\left(\frac{\tr{F}}{\tr{U}}\right) &,\quad\text{if}\qquad \tr{U} < \rho_{\mu,\mu_c}\\
    \palpha(F) \mp\,\arccos\left(\frac{\sradmm}{\tr{U}}\right)&,\quad\text{if}\qquad \tr{U} \geq \rho_{\mu,\mu_c}\;.
  \end{cases}
\end{equation}
\end{theo}

\begin{proof}
  By~\coref{cor:reduction:grioli}, $\polar(F)$ is uniquely optimal
  for the classical parameter range $\mu_c \geq \mu > 0$.
  The associated rotation angle $\palpha(F)$ is immediately obtained
  from the explicit formula for the polar factor given in~\coref{cor:polar_planar}.
  Let us now discuss the more delicate non-classical parameter regime
  $\mu > \mu_c \geq 0$. Here, the rescaling $\widetilde{F}_{\mu,\mu_c}$
  plays a decisive role. First, the condition $0< \tr{U} < \sradmm$ is
  equivalent to $0 < \tr{\widetilde{U}_{\mu,\mu_c}} < 2$. In this case,
  the polar factor $\polar(\widetilde{F}_{\mu,\mu_c}) = \polar(F)$ is uniquely
  optimal. For the parameter domain $\tr{U} \geq \sradmm$ which is equivalent
  to $\tr{\widetilde{U}_{\mu,\mu_c}} \geq 2$, however, we obtain
  the optimal relative rotation angles
  \begin{equation}
    \beta^{\pm}_{\mu,\mu_c}(F)
    = \beta^{\pm}_{1,0}(\widetilde{F}_{\mu,\mu_c})
    = \pm\arccos\left(\frac{\rho_{1,0}}{\tr{\widetilde{U}_{\mu,\mu_c}}}\right)
    = \pm\arccos\left(\frac{\rho_{\mu,\mu_c}}{\tr{U}}\right)\;.
  \end{equation}
  This gives $\alpha_{1,0}^\pm(\widetilde{F}_{\mu,\mu_c})
  = \palpha(\widetilde{F}_{\mu,\mu_c}) - \beta^{\pm}_{1,0}(\widetilde{F}_{\mu,\mu_c})
  = \palpha(F) - \beta^{\pm}_{1,0}(\widetilde{F}_{\mu,\mu_c})$.\qedhere
\end{proof}%
 %
\countres
\makeatletter{}%
\section{Optimal rotations for planar simple shear}
\label{sec:shear}
We now apply our previous optimality results to simple shear deformations.
Previously, in~\cite{Neff:2009:SSNC}, Neff and M\"unch contributed
the optimal planar rotations for simple shear. A {\bf simple shear of amount}
$\gamma \in \R$ is a homogeneous linear deformation represented by a matrix
of the form
\begin{equation}
  F_{\gamma} \eqdef
  \begin{pmatrix}
    1 & \gamma\\
    0 & 1
  \end{pmatrix}\;.
\end{equation}
In this section we derive the energy-minimizing rotation angles
$\malpha_{\mu,\mu_c}(\gamma) \eqdef \malpha_{\mu,\mu_c}(F_\gamma)$ for
simple shear.

Let us shortly consider the classical limit case $(\mu,\mu_c) = (1,1)$
which represents the entire classical parameter
range $\mu_c \geq \mu > 0$. Essentially due to~\theref{theo:intro:grioli}
the polar rotation $\polar(F_\gamma)$ is then uniquely optimal,
see~\coref{cor:reduction:grioli}. Thus, we proceed with the non-classical
limit case $(\mu,\mu_c) = (1,0)$ which represents the entire non-classical
parameter range $\mu > \mu_c \geq 0$, as we have seen
in~\leref{lem:parameter_reduction}.

Let us collect some properties of simple shear $F_\gamma$. We have
$\hsnorm{F_\gamma}^2 = 2 + \gamma^2$ and $\det{F_{\gamma}} = 1$, i.e.,
simple shear is volume preserving for any amount $\gamma$. This
allows us to compute
\begin{equation}
  \tr{U_\gamma} \;=\; \sqrt{\hsnorm{F_\gamma}^2 + 2\,\det{F_\gamma}}
               \;=\; \sqrt{4 + \gamma^2} \geq 2\;.
\end{equation}
Thus, the reduced energy always satisfies
$W^{\rm red}_{1,0}(F_\gamma) = W^{(3)}(F_\gamma)$ for simple shear $F_\gamma$, i.e.,
the non-classical branch is always optimal.
\begin{cor}[Optimal non-classical Cosserat rotations for simple shear]
  Let $(\mu,\mu_c) = (1,0)$ and let $F_{\gamma} \in \GL^+(2)$ be a simple
  shear of amount $\gamma \in \R$. Then,
  \begin{equation} \gamma \neq 0
    \quad\Longrightarrow\quad
    \rpolar^\pm_{1,0}(F_\gamma) \;\mathbf{\neq}\; \polar(F_\gamma)\;.
  \end{equation}
\end{cor}
\begin{proof}
  First, $\tr{U_\gamma} \geq 2$ for all $\gamma \in \R$
  and with~\theref{theo:char_rcrit10} the optimal
  relative rotation angle $\beta \in [-\pi,\pi]$ satisfies
  \begin{equation}
    \abs{\beta(U_\gamma)} = \arccos(\frac{2}{\tr{U_\gamma}})\quad \in [0,\pi]\;.
  \end{equation}
  For $\gamma \neq 0$, it is easy to see that $\tr{U_\gamma}^2 > 4$. Since
  $\arccos(2/x)$ is strictly increasing for $x \geq 2$, we finally conclude:
  \begin{equation*}
    0\; < \; \abs{\beta(F_{\gamma})} \;=\; \arccos(\frac{2}{\tr{U_\gamma}}) \;=\; \arccos(\frac{2}{\sqrt{4 + \gamma^2}})\;.\tag*{\hspace{-1em}\qedhere}
  \end{equation*}
\end{proof}

\begin{rem}[Symmetry of the first Cosserat deformation
            tensor $\mstretch$ in simple shear]
  A simple shear $F_\gamma$ by a non-zero amount $\gamma \neq 0$
  automatically generates an optimal microrotational response
  $\rpolar^\pm(F_\gamma)$ which deviates from the continuum
  rotation $\polar(F)$. This implies that the associated first
  Cosserat deformation tensor $\mstretch^\pm_{1,0}(F_\gamma) \eqdef \rpolar_{1,0}^\pm(F_\gamma)^TF_\gamma$ is not symmetric for any
  $\gamma \neq 0$; cf.~\leref{lem:mstretch_symmetry}.
\end{rem}

\begin{rem}[Consistency with~\cite{Neff:2009:SSNC}]
  It is not hard to show that the explicit minimizers
  $\rpolar^\pm_{1,0}(F_\gamma)$ for the optimal
  Cosserat rotations previously obtained do exactly match those
  computed in~\cite{Neff:2009:SSNC}[p. 12, Equation (3.24)]. We have
  found the following identity to be helpful for the
  verification: $\arctan(\frac{\gamma}{2})
  = \sign(\gamma)\,\arccos(\frac{2}{\sqrt{4 + \gamma^2}})
  = \sign(\gamma)\,\arccos(\frac{2}{\tr{U_\gamma}})$\;.
\end{rem}

\begin{figure}[t]
  \begin{center}
    \begin{tikzpicture}
      \node (Pic) at (0,0)
            {\includegraphics[width=10cm]{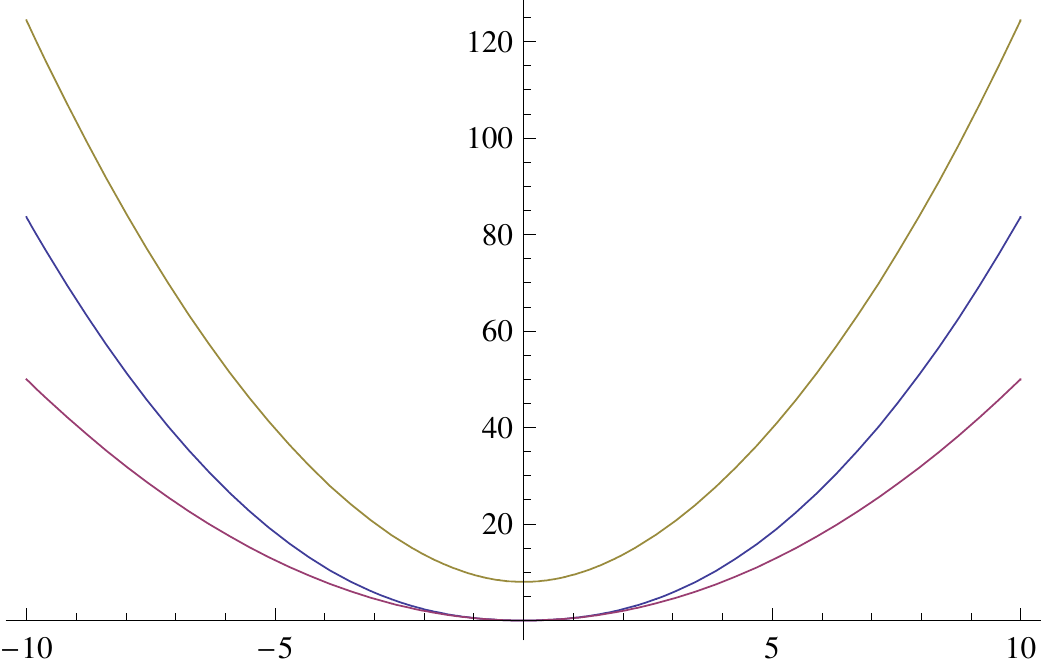}};
      \node (W1) at (-5.55,2.85) {$W^{(1)}(F_\gamma)$};
      \node (W2) at (-5.55, 1.05) {$W^{(2)}(F_\gamma)$};
      \node (W3) at (-5.55,-0.5) {$W^{(3)}(F_\gamma)$};
      \node (gamma) at (5.35,-2.8) {$\gamma$};
    \end{tikzpicture}
  \end{center}
  \begin{caption}[Critical energy levels $W^{(i)}$ for the case of simple shear.]%
    { \label{fig:W123_simple_shear}
      Plot of the critical energy levels $W^{(i)}(F_\gamma)$,~$i = 1,2,3$,
      of $\wsym$ for a simple shear $F_\gamma$ of amount $\gamma$.
      Note that \mbox{$W^{(1)} \geq W^{(2)} \geq W^{(3)}$}. The critical
      energy levels are realized by
      \mbox{$W^{(1)}\; \hat{=}\; -\polar(F_\gamma)$},
      \mbox{$W^{(2)}\; \hat{=}\; +\polar(F_\gamma)$} and
      \mbox{$W^{(3)}\; \hat{=}\; \rpolar^\pm_{1,0}(F_\gamma)$}, respectively.
    }
  \end{caption}
\end{figure}

\begin{figure}[H]
  \hspace{1.65cm}
  \begin{tikzpicture}
    \node (Pic) at (0,0)
          {\includegraphics[width=10cm]{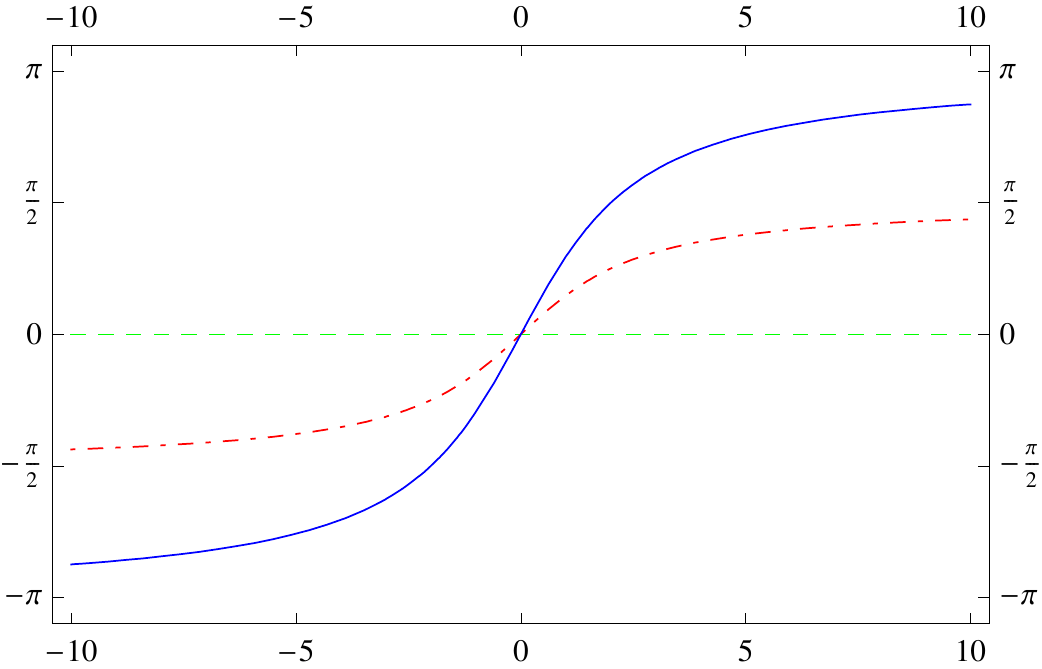}};
          \node (a1) at (3.05, 2.40) {$\malpha^{\,+}_{1,0}(\gamma)$};
          \node (a2) at (3.05, 1.35) {$\palpha(\gamma)$};
          \node (a3) at (3.05, 0.25) {$\malpha^{\,-}_{1,0}(\gamma)$};
          \draw[->,line width = 0.1pt] (-0.5,-3.75) -- (-1,-3.75);
          \node (gamma) at (0,-3.75) {$\gamma$};
          \draw[->,line width = 0.1pt] (0.5,-3.75) -- (1,-3.75);
          \draw[->,line width = 0.1pt] (-5.5,-0.5) -- (-5.5,-1);
          \node (gamma) at (-5.5,0) {$\alpha(\gamma)$};
          \draw[->,line width = 0.1pt] (-5.5,0.5) -- (-5.5,1);
  \end{tikzpicture}
  \begin{caption}[Optimal microrotation angles for $\wsym$ and simple shear.]%
    {Plot of the optimal microrotation angles $\malpha^\pm_{1,0}(\gamma)$
      for $\wsym$ and simple shear $F_\gamma$ of amount $\gamma \in \R$. The negative
      optimal branch $\malpha^{-}_{1,0}(\gamma)$ [dashed green curve]
      exactly eliminates the angle $\palpha(F)$ and vanishes identically.
      The positive branch $\malpha^{+}_{1,0}(\gamma)$
      [continuous blue curve] corresponds to a rotation by
      $2\,\palpha(\gamma)$ [dot-dashed red curve]. Note the symmetry
      w.r.t. to the continuum rotation angle $\palpha(\gamma)$}
    \label{fig:alphamin_simple_shear}
  \end{caption}
\end{figure}

\subsection{Simple glide and cancellation of the polar factor}
For the case of simple shear, one of the optimal Cosserat rotations for
the shear-stretch energy $\wsym$ exactly cancels
the polar factor. More precisely, one of the two rotations
$\rpolar^\pm(F_\gamma)$ is the identity element $\id_2 \in \SO(2)$, while
the other solution is given by $\polar(F_\gamma)^2$, i.e., a
rotation by $2\,\palpha(F_\gamma)$, see also~\figref{fig:alphamin_simple_shear}.
It would be quite intriguing if this ``gliding'' behavior
were specific to simple shear, but, as it turns out, it is possible to
construct other examples showing the same behavior.

To see this, note first that
\begin{equation}
  \abs{\palpha} \stackrel{\text{(Cor. \ref{cor:polar_planar}})}{=} \arccos\left(\frac{\tr{F}}{\tr{U}}\right),
  \quad\text{and}\quad
  \abs{\beta} \stackrel{\text{(Thm. \ref{theo:malpha10pm})}}{=} \arccos\left(\frac{2}{\tr{U}}\right)\;.
\end{equation}
The condition that one of the relative rotations cancels the continuum
rotation is given by
\begin{equation}
    \abs{\beta} = \abs{\palpha}
    \quad\isequivto\quad
    \frac{\tr{F}}{\tr{U}} = \frac{2}{\tr{U}}\;,\label{eq:loc:cancel_polar_cond}
\end{equation}
whenever $\tr{U} \geq 2$. Hence the set of
matrices for which the polar rotation is canceled by
a minimizing relative rotation is given by
$\setdef{F \in \GL^+(2)}{\tr{F} = 2 \land \tr{U} \geq 2}$.
This set is nonempty, because it contains the family of simple
shears $F_{\gamma}$. In order to see that this set
also contains homogeneous deformations which are not simple
shears, we consider the matrix
\begin{equation}
  \id^* \eqdef
  \begin{pmatrix}
    1 & 0\\
    0 & -1
  \end{pmatrix}
\end{equation}
and set
\begin{equation}
  F_{\gamma,\kappa} \eqdef F_\gamma + \kappa\id^*
  \quad\text{and}\quad
  U_{\gamma,\kappa} \eqdef   \sqrt{F_{\gamma,\kappa}^T\,F_{\gamma,\kappa}}\;.
\end{equation}
Note that $\det{F_{\gamma,\kappa}} = 1 - \kappa^2$ implies that
$F_{\gamma, \kappa} \in \GL^+(2)$ for $0 < \abs{\kappa} < 1$. Further,
\begin{equation}
  \forall \gamma,\kappa \in \R: \quad \tr{F_{\gamma,\kappa}} = \tr{F_\gamma}
  \quad\text{and}\quad
  \tr{JF_{\gamma,\kappa}} = \tr{JF_\gamma}\;,
\end{equation}
which implies
\begin{equation}
  \tr{U_{\gamma,\kappa}}^2
= \tr{F_{\gamma,\kappa}}^2 + \tr{J(F_{\gamma,\kappa})}^2
= \tr{F_\gamma}^2 + \tr{JF_\gamma}^2  = \tr{U_\gamma}^2\;.
\end{equation}
Hence, both quantities $\tr{U_{\gamma,\kappa}}$ and $\tr{F_{\gamma,\kappa}}$
are independent of $\kappa$ and condition~\eqref{eq:loc:cancel_polar_cond}
is automatically satisfied for all admissible $F_{\gamma,\kappa}$, i.e.,
for $0 < \abs{\kappa} < 1$.

We conclude that, given any simple shear $F_\gamma$ of amount $\gamma \in \R$,
there is a one parameter family $F_{\gamma,\kappa}, 0 < \abs{\kappa} < 1$
of matrices that are not simple shears for which one of the
optimal relative rotations $\hat{R}$ exactly cancels the
continuum rotation $\polar(F)$. The interesting ``glide behavior''
observed in \secref{sec:shear} is not specific to simple shear.

 %
\countres
\makeatletter{}%
\section{Conclusion}

In~\secref{sec:reduction}, we have seen that it is sufficient
to construct energy-minimizing rotations for the classical
limit-case $(\mu,\mu_c) = (1,0)$ and the non-classical limit
case $(\mu,\mu_c) = (1,1)$, respectively. For the classical parameter
range $\mu_c \geq \mu > 0$, the unique minimizing rotation
for $W_{\mu,\mu_c}(\mrot\,;F)$ is given by $\polar(F)$, in any dimension
$n$. For $\mu = \mu_c$, the reduced Cosserat shear energy can be
formally reduced to
\begin{equation}
  W^{\rm red}_{\mu,\mu}(F) = W_{\mu,\mu}(\polar(F)\,;F) = W_{\rm Biot,\mu,0}(F)\;.
\end{equation}
Hence, setting the Cosserat curvature coefficient $L_c = 0$, one can
expect the full quadratic Cosserat model to behave essentially like
a classical Biot model, see, e.g., the introduction to~\cite{Neff_Biot07}.

However, a fundamental motivation to use extended continuum models
such as Cosserat models, is to generate solutions showing non-classical
effects. For the quadratic Cosserat model (without curvature),
this is the case if there is a deviation $R \neq \polar(F)$, since
the model (formally) reduces to the well-known
Biot energy otherwise. In~\secref{sec:minimization}, we have
shown that this is only to be expected for the non-classical
parameter range $\mu > \mu_c \geq 0$. If non-classical solutions
should be generated already in the identity $\id_2$, then we even
have to require $\mu_c = 0$, since $\rho_{\mu,\mu_c} > 2$ otherwise.
The existence of the presented non-classical minimizers
$\rpolar(F)$ is highly interesting. In strong contrast,
if we replace the non-symmetric strain tensor $\mstretch - \id$
by $\log\,\mstretch$ in~\probref{intro:prob:weighted}, which is natural
in view of the Cartan decomposition of $\gl(n)$, one can show
that \emph{no} non-classical solutions exist for arbitrary
$\mu > 0$ and $\mu_c \geq 0$
\begin{equation}
  \argmin{\mrot\,\in\,\SO(n)}{%
    \left\{
    \mu\, \hsnorm{\sym\log(\mrot^TF)}^2
    \,+\,
    \mu_c\,\hsnorm{\skew\log(\mrot^TF)}^2 \right\}
    = \{\polar(F)\}
  }\;.
\end{equation}

For a proof and a deep discussion of the nature and properties
of logarithmic strain measures,
see~\cite{Neff:2015:GLS,Lankeit:2014:MML,Neff:2014:LMP}.

In our introduction, we have stated that the solution
to~\probref{intro:prob:weighted} for $n \geq 3$ is currently out
of reach. However, we have successfully computed non-classical
critical Cosserat microrotations for $n = 3$ using a parametrisation
by unit quaternions and computational algebra. Further, we
have managed to select the energy minimal branches, experimentally.
An extensive numerical validation shows, moreover, that our
candidates are very likely the global minimizers. The
mechanisms discovered for the case $n = 2$ in the present work
do carry over to the case $n = 3$ quite literally up to the
determination of the microrotation axis. This is the content
of a forthcoming second part of this paper~\cite{Fischle:2015:QC3D}.
For dimensions $n > 3$, the weighted~\probref{intro:prob:weighted}
is, to the best of our knowledge, still completely open. It seems
to us, however, reasonable to guess that a transformation into the
principal directions of stretch, i.e., the eigendirections of $U$,
is a good plan of attack.

\begin{rem}[Final Conclusion]
  To ascertain the complete absence of a non-classical response within
  a geometrically nonlinear quadratic Cosserat-micropolar
  shear-stretch energy, one must choose a classical
  parameter set, i.e., $\mu_c \geq \mu > 0$.
\end{rem}

%
%
 %
\countres

\addcontentsline{toc}{section}{References}
\bibliographystyle{plain}
\bibliography{./literaturNeff,./fischle}

\appendix
\renewcommand{\theequation}{\Alph{section}.\arabic{equation}}
\makeatletter{}%
\section{Appendix}

\subsection{Some planar matrix identities for $U$ and $F$}
Let $n = 2$. Applying the trace to both sides of the Cayley-Hamilton equation
and exploiting linearity, we obtain
\begin{equation}
  \tr{X^2 - \tr{X}\,X + \det{X}\,\id_2} = \tr{0}
  \quad\isequivto\quad
  \tr{X^2} = \tr{X}^2 - 2\,\det{X}\;.\label{eq:square_trace_id}
\end{equation}
This leads us to the following identity.
\begin{lem}
  \label{lem:trU}
  Let $F \in \GL^+(2)$ and $U \eqdef \sqrt{F^TF} \in \Psym(2)$. Then,
  \begin{align}
    \tr{U} &= \sqrt{\hsnorm{U}^2 + 2\,\det{U}} = \sqrt{\hsnorm{F}^2 + 2\,\det{F}}\;.
  \end{align}
\end{lem}
\begin{proof}
  Note first that $\tr{U^2} = \scalprod{U}{U} = \hsnorm{U}^2$ and
  that $\tr{U} > 0$. The expression in terms of $F \in \GL^+(2)$
  is implied by $\det{F} = \det{\polar(F)^TF} = \det{U}$
  and $\hsnorm{F}^2 = \hsnorm{\polar(F)^TF}^2 = \hsnorm{U}^2$.\qedhere
\end{proof}%
%
%
%
%
%
%
%
%
%
%
%
%
%
%
%
%
%
%
%
%
%
%
%
%

%

%
%
%
%
%
%
%
%
%
%
%

%

%
%

%
%
%
%
%
%
%
%
%
%
%
%
%
%
%
%
%
%
%
%
%
%
%
%
%
%
%
%
%

%
%
%
%
%
%
%
%
%
%
%
%
%
%
%
%
%

%
%
%
%
%
%
%
%
%
%
%
%
%
%
%
%
%
%
%
%
%
%
%
%
%
%
%
%
%
%
%
%
%
%
%
%
%
%
%
%
%
%
%
%
%
%
%
%
%
%
%
%
%
%
%
%
%

%
%
%
%
%
%
%
%
%
%
%
%
%
%
%
%
%
%
%
%
%
%
%
%
%
%
%
%
%
%

%
%
%
%
%
%
%
%
%
%
%
%
%
%
 %
\countres

\end{document}